\documentclass[11pt,reqno,a4paper]{amsart}
\usepackage{amssymb}
\usepackage{mathrsfs}
\usepackage{latexsym}
\usepackage{mathtools}
\usepackage{cite}
\usepackage{hyperref}
\theoremstyle{plain}
\newtheorem{theorem}{Theorem}

\newtheorem{lemma}{Lemma}

\newtheorem{proposition}{Proposition}
\newtheorem*{TA}{Theorem A}
\newtheorem*{TB}{Theorem B}

\theoremstyle{definition}

\theoremstyle{remark}
\newtheorem{remark}{Remark}
\newcommand{\sgn}{\text{sgn}}
\newcommand{\e}{\epsilon}

\newcommand{\R}{\mathbb R}
\newcommand{\T}{\mathbb T}
\newcommand{\N}{\mathbb N}
\newcommand{\Z}{\mathbb Z}

\newcommand{\vertiii}[1]{{\left\vert\kern-0.25ex\left\vert\kern-0.25ex\left\vert #1 
    \right\vert\kern-0.25ex\right\vert\kern-0.25ex\right\vert}}

\newenvironment{bcases}
  {\left\lbrace\begin{aligned}}
  {\end{aligned}\right\rbrace}
\numberwithin{equation}{section}

\begin{document}

\title[Control of Fifth Order KdV]{Control and Stabilization of the
Periodic Fifth Order Korteweg-de Vries Equation}
\author{Cynthia Flores}
\address[C. Flores]{California State University, Channel Islands \\
Bell Tower East 2762 \\
One University Drive \\
Camarillo, CA 93012}
\email{cynthia.flores@csuci.edu}

\author{Derek L. Smith}
\address[D. L. Smith]{B\^atiment des Math\'ematiques\\
EPFL\\
Station 8 \\
CH-1015 Lausanne\\
Switzerland.}
\email{derek.smith@epfl.ch}
\begin{abstract} 
We establish local exact control and local exponential
stability of periodic solutions of fifth order
Korteweg-de Vries type equations in $H^s(\T)$, $s>2$.
A dissipative term is incorporated into the control
which, along with a propagation of regularity
property, yields a smoothing effect permitting the
application of the contraction principle.
\end{abstract}
\maketitle

\begin{section}{Introduction}\label{S:1}

We study control of the fifth order Korteweg de-Vries (KdV) equation
\begin{equation} \label{E:KDV5}
\partial_tu - \partial_x^5u
	- 30u^2\partial_xu + 20\partial_xu\partial_x^2u
	+ 10u\partial_x^3u = 0
\end{equation}
where $u=u(x,t)$ denotes a real-valued function.
This equation appears in the sequence of nonlinear
dispersive equations
\begin{equation} \label{E:KDVH}
\partial_tu + \partial_x^{2j+1}u + Q_j(u,\partial_xu,\dots,\partial_x^{2j-1}u) = 0,
	\quad j\in\mathbb{Z}^+,
\end{equation}
known as the KdV hierarchy.
The specification of the polynomials $Q_j$ arises
from the observation in \cite{Gardner1967} that
the eigenvalues of the Schr\"odinger operator
$L(u)=\frac{d^2}{dx^2}-u(x,\cdot)$ are independent
of time when $u$ evolves as a solution to the usual
KdV equation
\begin{equation} \label{E:KDV}
\partial_tu + \partial_x^3u + u\partial_xu = 0.
\end{equation}
By imposing a Lax pair structure
\begin{equation*}
\partial_tu = [B_j;L(u)]
\end{equation*}
the same statement holds for any equation in the
sequence \eqref{E:KDVH} when $B_j$ is a
skew-symmetric operator chosen so that
$[B_j;L(u)]$ has degree zero \cite{MR0235310}.
The resulting hierarchy consists of a family of completely
integrable equations which can be solved by the
inverse scattering method.

This paper focuses on a fifth order equation
generalizing \eqref{E:KDV5}
\begin{align} \label{E:K5G}
&\partial_tu - \partial_x^5u
	+ \beta_0\partial_x^3u + \beta_1\partial_xu \notag \\
	&\qquad\qquad
	+ c_0u\partial_xu + c_1u^2\partial_xu
	+ c_2\partial_xu\partial_x^2u + c_3u\partial_x^3u = 0,
\end{align}
where $\beta_0,\beta_1,c_0,c_1,c_2,c_3$
are real constants.
Solutions to this equation formally conserve volume
and arise in a number of physical situations.
With $c_1=0$, this equation was shown to model the
water wave problem for long, small amplitude waves
over a shallow bottom \cite{MR755731};
see also \cite{MR2177163}.
The family \eqref{E:K5G} also contains Benney's
model of short and long wave interaction \cite{MR0463715}.
Letting $c_0=1$ and $c_1=c_2=c_3=0$ yields the
Kawahara equation
\begin{equation} \label{E:KAWA}
\partial_tu - \partial_x^5u
	+ \beta_0\partial_x^3u + \beta_1\partial_xu
	+ u\partial_xu = 0,
\end{equation}
which describes the propagation of magneto-acoustic
waves in a plasma \cite{Kawahara1972}.

The initial value problem (IVP) associated to these
equations is naturally studied in the Sobolev scale
\begin{equation*}
H^s(\mathbb{K})
    = (1-\partial_x^2)^{-s/2} L^2(\mathbb{K}), \qquad s\in\R,
\end{equation*}
$\mathbb{K}=\T$ or $\R$.
For equations \eqref{E:KDV5}, \eqref{E:KDV} and
\eqref{E:KAWA}, the problem of determining the minimal
Sobolev regularity required to ensure well-posedness of
the associated IVP has been studied extensively.
The KdV equation is globally well-posed
in $H^s(\R)$ for $s\geq-3/4$ and $H^s(\T)$ for $s\geq-1/2$.
We mention the works 
\cite{MR1211741}, \cite{MR1209299,MR1215780},
\cite{MR1969209}, \cite{MR2531556} and \cite{MR2501679}
in this regard.
For the Kawahara equation we have
local well-posedness in $H^s(\R)$ for $s\geq-2$
\cite{MR2767079} and global well-posedness in $H^s(\R)$
for $s>-38/21$ \cite{MR2989690}.
In the periodic setting, the equation is locally well-posedness
in $H^s(\T)$ for $s\geq-3/2$ and globally well-posedness
in $H^s(\T)$, $s\geq-1$ \cite{MR2989690}.
Ponce \cite{MR1216734} established local well-posedness
for the fifth order equation \eqref{E:K5G} in
$H^s(\R)$, $s\geq4$, using sharp linear estimates and the
Bona-Smith argument \cite{MR0385355}.
Kenig, Ponce and Vega investigated a class of equations
containing the KdV hierarchy in polynomially weighted
Sobolev spaces by combining a commuting vector
field identity with the contraction principle
\cite{MR1321214,MR1195480}.
Pilod \cite{MR2446185} showed that for each $j\geq2$,
the solution map
$H^s(\mathbb{R}) \ni u_0
	\mapsto u \in C([0,T];H^s(\mathbb{R}))$
corresponding to the IVP for the equation
\begin{equation*}
\partial_tu + \partial_x^{2j+1}u \pm u\partial_x^ku = 0,
	\qquad k>j,
\end{equation*}
is not $C^2$ at the origin for any $s\in\mathbb{R}$.
In fact, it is not even uniformly continuous \cite{MR2455780}.
Thus, in contrast to the KdV equation, higher order members
of the KdV hierarchy \eqref{E:KDVH} cannot be solved using the
contraction principle in $H^s(\R)$ alone.
Kwon \cite{MR2455780} introduced a corrected energy
and refined Strichartz estimate to establish local
well-posedness for the fifth order KdV equation
\eqref{E:KDV5} in $H^s(\R)$, $s>5/2$.
Kenig and Pilod \cite{MR3513119} applied this technique to
a class of equations containing the KdV hierarchy,
obtaining local well-posedness in $H^s(\R)$
for $s>4j-9/2$, $j\geq2$.
Global well-posedness of the fifth order KdV equation in
$H^2(\R)$ was established simultaneously by Guo, Kwak and
Kwon \cite{MR3096990} and Kenig and Pilod \cite{MR3301874} using
a Bourgain space approach.
In the periodic setting, Schwarz \cite{MR761761} obtained
the existence of weak solutions to equations in the KdV
hierarchy \eqref{E:KDVH} corresponding to data in $H^s(\T)$,
$s\geq2j+1$ and $s\in\Z^+$, with uniqueness holding under
the condition $s\geq3j+1$.
Recently, Kwak \cite{MR3473453} obtained global
well-posedness of equation \eqref{E:KDV5} in $H^2(\T)$.
The proof relies somewhat on the completely
integrable structure of the equation.

For the forced fifth order equation
\begin{align} \label{E:K5F}
&\partial_tu - \partial_x^5u
	+ \beta_0\partial_x^3u + \beta_1\partial_xu \notag \\
	&\qquad\qquad\qquad
	+ c_0u\partial_xu + c_1u^2\partial_xu
	+ c_2\partial_xu\partial_x^2u + c_3u\partial_x^3u = F(x,t),
\end{align}
on a periodic domain, we investigate two questions
central to control theory.

\vspace{\baselineskip}

{\bf Exact Control Problem}
Given an initial state $u_0$ and terminal state $u_1$,
can one find an appropriate forcing function $F$ so that
equation \eqref{E:K5F} admits a solution $u=u(x,t)$
which satisfies $u(\cdot,0)=u_0$ and $u(\cdot,T)=u_1$?

\vspace{\baselineskip}

{\bf Stabilization Problem}
Can one find a feedback law $F=Ku$ so that the
resulting closed-loop system is asymptotically stable
as $t\rightarrow+\infty$?

\vspace{\baselineskip}

As solutions to \eqref{E:K5F} satisfy the identity
\begin{equation*}
\frac{d}{dt} \int_\T u(x,t) dx = \int_\T F(x,t) dx,
\end{equation*}
we achieve conservation of volume by
choosing $F$ of the form
\begin{equation} \label{CTRL}
(Gh)(x,t) := g(x)\left(h(x,t) - \int_\T g(y)h(y,t) dy\right),
\end{equation}
where $g \in C^\infty(\T)$ is nonnegative, has the mean
value property
\begin{equation*}
2\pi[g] = \int_\T g(x)dx = 1,
\end{equation*}
and is allowed to be supported in a proper subinterval
of the torus.

Russell and Zhang obtained  the first control results for
an equation in the KdV hierarchy.
Using the smoothing effect discovered by Bourgain
\cite{MR1209299,MR1215780} they established
{\em local exact controllability} for the
KdV equation \eqref{E:KDV}.

\begin{TA} \cite{MR1360229}
Let $T>0$ and $s\geq0$ be given.
Then there exists $\delta>0$ such that for any $u_0,u_1 \in H^s(\T)$
with $[u_0]=[u_1]$ and
\begin{equation*}
\|u_0\|_s + \|u_1\|_s \leq \delta,
\end{equation*}
one can find a control $h \in L^2(0,T,H^s(\T))$ such that the equation
\begin{equation*}
\partial_tu + \partial_x^3u + u\partial_xu = Gh
\end{equation*}
has a solution $u \in C([0,T];H^s(\T))$ satisfying
\begin{equation*}
u(\cdot,0)=u_0  \quad\text{and}\quad  u(\cdot,T)=u_1.
\end{equation*}
\end{TA}
Additionally, they proved {\em local exponential stability}.
\begin{TB} \cite{MR1360229}
Let $\kappa>1$ and $s=0$ or $s\geq1$ be given.
Then there exists positive constants, $C$, $\delta$ and $\lambda$ such
that if $u_0 \in H^s(\T)$ with $\|u_0-[u_0]\|_s\leq\delta$, then the
corresponding solution $u$ of the system
\begin{equation*}
\partial_tu + \partial_x^3u + u\partial_xu = -\kappa Gu
\end{equation*}
satisfies
\begin{equation*}
\|u(\cdot,t)-[u_0]\|_s \leq Ce^{-\lambda t}\|u_0-[u_0]\|_s
\end{equation*}
for all $t\geq0$.
\end{TB}

Laurent, Rosier and Zhang \cite{MR2753618} later proved
{\em global exact controllability} and {\em global
exponential stability} for the KdV equation
\eqref{E:KDV} in $H^s(\T)$, $s\geq0$.
Their technique relied on the structure of the commutator
$[\phi;\partial_x^3]$ for $\phi \in C^\infty(\T)$.
On the real line, Kato \cite{MR759907} utilized this structure
to conclude that the solution to \eqref{E:KDV} corresponding
to data $u_0 \in H^s(\R)$ lies in $L^2(0,T;H_{\text{loc}}^{s+1}(\R))$.
Though such a smoothing effect is false on the torus,
the same formal computation reveals that if the solution
lies in $L^2(0,T;H^{s+1}(\Omega))$ for some open set
$\Omega\subset\T$, then one may conclude that it also
lies in $L^2(0,T;H^{s+1}(\T))$.
This propagation of regularity, along with a similar
propagation of compactness lemma, was previously applied
to control of wave \cite{MR2253466}
and Schr\"odinger equations \cite{MR2654198,MR2644360}.
Though not discussed further here, we mention the extensive
work on the control theory of the Korteweg-de Vries equation
on a bounded domain, a review of which may be found in
the survey \cite{MR2565262}.

We now discuss control results pertaining to equation
\eqref{E:K5G}.
The situation is most developed for the Kawahara equation;
in particular, Zhao and Zhang \cite{MR3356494} applied the
method of \cite{MR2753618} to obtain global exact control
and global exponential stability for periodic solutions in
$H^s(\T)$, $s\geq0$.
Moreover, exponential stability has been demonstrated for
the initial-boundary value problem associated to the
Kawahara equation \eqref{E:KAWA} on an interval in a
number of situations.
In the presence of a feedback term $F=-g(x)u$, we mention
the works \cite{MR2775188} and \cite{MR2834850}.
Without a feedback term, see \cite{MR2434910} for the
case of zero boundary conditions and \cite{MR3462574}
for a boundary dissipation mechanism.

In the case of $c_2^2+c_3^2>0$,
Glass and Guerrero \cite{MR2569891} established local
controllability to trajectories for the boundary value problem
associated to equation \eqref{E:K5G} by using Carleman
estimates and a smoothing effect of Kato type derived
from the boundary conditions.
To the best of our knowledge, there are no results
concerning the exact control or exponential stability
of equation \eqref{E:K5G} in the periodic setting
when ${c_2^2+c_3^2>0}$.
  
In this paper, we present affirmative answers to the exact
control and stabilization problems for equation
\eqref{E:K5F} on a periodic domain.
To overcome the lack of an adequate smoothing effect,
we adopt the approach of \cite{MR3335395}.
To stabilize \eqref{E:K5F} we consider a feedback law
\begin{equation*}
F = -GD^3Gu.
\end{equation*}
In the linear homogeneous case, scaling the resulting
equation
\begin{equation} \label{E:LIN}
\partial_tv - \partial_x^5v
	+ \beta_0\partial_x^3v + \beta_1\partial_xv
	+ GD^3Gv = 0
\end{equation}
by
$v$ yields
\begin{equation*}
\frac12 \|v(T)\|_{L^2(\T)}^2
	+ \int_0^T \|D^{3/2}(Gv)\|_{L^2(\T)}^2 \, dt
		= \frac12 \|v_0\|_{L^2(\T)}^2,
\end{equation*}
which suggests a gain of $3/2$ derivatives in the control
region ${\omega=\{x\in\T : g(x)>0\}}$.
Using a propagation of regularity property we conclude
\begin{equation} \label{SMOOTH}
\|v\|_{L^2(0,T;H^{3/2}(\T))} \leq c(T;\|v_0\|_{L^2(\T)}).
\end{equation}
By considering the forcing term $F=-GD^3Gv+Gh$, a similar
smoothing effect holds and we obtain exact controllability
of the resulting linear equation by a classical observability
argument.
Thus a contraction principle argument yields the following
nonlinear result.

\begin{theorem} \label{T:CONT}
Let $T>0$ and $s>2$ be given.
Then there exists $\delta>0$ such that for any
$u_0,u_1 \in H^s(\T)$ with $[u_0]=[u_1]$ and
\begin{equation*}
\|u_0\|_s + \|u_1\|_s \leq \delta,
\end{equation*}
one can find a control $h \in L^2(0,T,H^{s-3/2}(\T))$
such that the equation \eqref{E:K5F} with
$F=-GD^3Gu+Gh$ has a solution $u \in C([0,T];H^s(\T))$ satisfying
\begin{equation*}
u(\cdot,0)=u_0  \quad\text{and}\quad  u(\cdot,T)=u_1.
\end{equation*}
\end{theorem}

\noindent
Similarly, the linear exponential stability of solutions
to \eqref{E:LIN} combined with the contraction principle
in an appropriate space yields exponential stability for
the nonlinear problem.

\begin{theorem} \label{T:STAB}
Let $s>2$ be given.
Then there exists constants $\rho,\lambda,C>0$,
such that for any $u_0 \in H^s(\T)$
with ${\|u_0-[u_0]\|_s\leq\rho}$, equation \eqref{E:K5F}
with $F=-GD^3Gu$ admits a unique solution satisfying
\begin{equation*}
u \in C([0,T];H^s(\T)) \cap L^2(0,T;H^{s+3/2}(\T))
\end{equation*}
for any $T>0$ and such that
\begin{equation*}
\|u(t)-[u_0]\|_s \leq Ce^{-\lambda t}\|u_0-[u_0]\|_s.
\end{equation*}
\end{theorem}

We note that Theorem \ref{T:CONT} inherits the limitation of
\cite{MR3335395} in that the control $h$ is realized in
the space $L^2(0,T;H^{s-3/2}(\T))$ instead of
$L^2(0,T;H^s(\T))$.

The remainder of the paper is organized as follows.
Section \ref{S:2} contains estimates for the linear
problem \eqref{E:LIN}.
The proof of Theorem \ref{T:STAB} is found in
Section \ref{S:3} and the proof of Theorem \ref{T:CONT}
is found in Section \ref{S:4}.
Additionally, Section \ref{S:3} describes how to extend these
results to a family of equations containing the KdV hierarchy.

\end{section}

\begin{section}{Preliminaries and Linear Estimates}\label{S:2}

Throughout the sequel, it suffices to consider only
the case $[u_0]=0$; the change of dependent variable
$\tilde{u}=u-[u_0]$ in equation \eqref{E:K5G} leads
to an equation in $\tilde{u}$ of the same type.
We denote $H_0^s(\T) = \{u \in H^s(\T) : [u]=0 \}$.
The usual $L^2(\T)$ scalar product is written
$(u,v) = \int_\T u(x)v(x) \, dx$ and in $H^s(\T)$, $s\in\R$,
$(u,v)_s=((1-\partial_x^2)^{s/2}u,(1-\partial_x^2)^{s/2}v)$.
The norm in $H^s(\T)$ is given by
$\|u\|_s=(u,u)_s^{1/2}$ where we abbreviate $\|u\|_0=\|u\|$.
It is convenient to define the operator $D^r$, $r\in\R$, as
\begin{equation}
\widehat{D^ru}(k) =
\begin{cases}
	|k|^r\hat{u}(k)	& \text{if $k\neq0$} \\
	\hat{u}(0)	& \text{if $k=0$},
\end{cases}
\end{equation}
so that $\|u\|_s\cong\|D^su\|$ for $u \in H_0^s(\T)$.
This operator satisfies the following commutator estimate.

\begin{lemma} \cite[Lemma A.1]{MR2654198}
If $\psi \in C^\infty(\T)$, then for any $r,s\in\R$
\begin{equation} \label{COMM}
\|D^r[D^s;\psi]f\| \leq c(r;s;\psi)\|f\|_{r+s-1}.
\end{equation}
\end{lemma}
We will make use of the Hilbert transform $\mathcal H$,
defined as a Fourier multiplier via the formula
\begin{equation*}
\widehat{\mathcal H f}(k) = -i \, \sgn(k) \hat{f}(k),
    \qquad k\in\Z.
\end{equation*}

In addition to being volume-preserving and self-adjoint
on $L^2(\T)$, one sees using \eqref{COMM} that the
operator $G$ is a bounded operator on $H^s(\T)$ for
any $s\in\R$.
The definition \eqref{CTRL} yields for $\psi \in C^\infty(\T)$
\begin{align} \label{CTRL1}
G(\psi h)
	&= \psi\,Gh
			+ g\left(\psi\int gh \,dy
			- \int \psi gh \,dy\right) \notag\\
	&=: \psi\,Gh + \tilde{h}
\end{align}
where $\|D^s\tilde{h}\| \leq c(s;g;\psi)\|h\|$ for any $s\in\R$.
Similarly,
\begin{equation} \label{CTRL2}
[D^s;G]D^rf = [D^s;g]D^rf - D^sg\int fD^rg + g\int fD^{r+s}g
\end{equation}
for any $s,r\in\R$.
Thus, writing $r=r_1+r_2$,
\begin{equation} \label{CTRL3}
\|D^{r_1}[D^s;G]D^{r_2}f\| \leq c(r;s;g)\|f\|_{r+s-1}.
\end{equation}

Next, we shall deduce estimates of solutions
to the linear problem for $\epsilon>0$. 
\begin{equation} \label{I:LINEAR}
\begin{cases}
	\partial_tv + (\epsilon D^5 - \partial_x^5)v
	+ \beta_0\partial_x^3v + \beta_1\partial_xv + GD^3Gv = F,
		\qquad x\in\T, t\geq0, \\
	v(x,0) = v_0(x).
\end{cases}
\end{equation}
As it does not affect the analysis we assume $\beta_0=\beta_1=0$.
We first uncover apriori $H_0^s(\T)$ bounds on smooth solutions to the
above IVP by incorporating a propagation of regularity argument in
the same vein as \cite{MR2253466}, \cite{MR2654198} and \cite{MR2644360}.
Writing $w=D^sv$ with $v$ smooth, we see that $w$ solves
\begin{equation*}
\partial_tw + (\epsilon D^5 - \partial_x^5)w + GD^3Gw + Ew = D^sF
\end{equation*}
where the ``remainder" operator
\begin{equation} \label{EROP}
E = GD^3[D^s;G]D^{-s} + [D^s;G]D^3GD^{-s}
\end{equation}
has order 2.
The following weighted energy identity will be utilized.
\begin{lemma}
A smooth solution $v=v(x,t)$ to IVP \eqref{I:LINEAR} satisfies
\begin{align} \label{WGHT}
&\frac12\frac{d}{dt} \int w^2 \psi \, dx
		+ \frac52 \int (\partial_x^2w)^2 \psi' \, dx \notag\\
	&\qquad\qquad
		+ \int D^{3/2}(Gw) D^{3/2}G(\psi w) \, dx
		+ \int w Ew \, \psi \, dx \notag\\
	&\qquad\qquad
		+ \epsilon \left\{ \int (D^{5/2}w)^2 \psi \, dx
			+ \int D^{5/2}w[D^{5/2};\psi]w \, dx \right\} \notag\\
	&\qquad\qquad\qquad
		= \frac52 \int (\partial_xw)^2 \psi^{(3)} \, dx + \frac12 \int w^2 \psi^{(5)} \, dx + \int wD^sF \, \psi \, dx,
\end{align}
where $w=D^sv$ and $\psi
 \in C^\infty(\T)$.
\end{lemma}

Motivated by the gain of $3/2$-derivatives suggested by the form of
the control term, we study solutions to the IVP \eqref{I:LINEAR} in the spaces
\begin{equation}
Z_{s,T} = C(0,T;H_0^s(\T)) \cap L^2(0,T;H_0^{s+3/2}(\T)),
\end{equation}
with $s\in\R, T>0$, endowed with the norm
\begin{equation}
\|v\|_{s,T} = \|v\|_{L^\infty(0,T;H^s(\T))} + \|v\|_{L^2(0,T;H^{s+3/2}(\T))}.
\end{equation}
The next proposition establishes $\epsilon$-uniform bounds in $Z_{s,T}$.

\begin{proposition} \label{P:UNI}
Let $s\in\R$ and $0<\epsilon<1$.
A smooth solution to IVP \eqref{I:LINEAR} corresponding to data
$v_0 \in H_0^s(\T)$ satisfies $v \in Z_{s,T}$ with
\begin{equation} \label{UNIFORM}
\|v\|_{s,T} \leq c(s,T)\left( \|v_0\|_s + \|F\|_{L^2(0,T;H^{s-3/2}(\T))} \right)
\end{equation}
for any $T>0$ and $c(s,T)$ nondecreasing in $T$.
\end{proposition}
\begin{proof}
We show the details for the case $s=0$ and demonstrate the
necessary modifications when $s\neq0$.

\noindent
{\it (Case $s=0$.)} \\
In order to justify the following computations, assume
$v_0 \in H_0^5(\T)$ and $F \in C([0,T];H_0^5(\T))$ so that
\begin{equation*}
v \in C([0,T];H_0^5(\T)) \cap C^1([0,T];H_0^0(\T)).
\end{equation*}
Scaling the equation \eqref{I:LINEAR} by $v$, and for all $\tau<T$,
\begin{equation*}
\frac12 \|v(\tau)\|^2
	+ \int_0^\tau \|D^{3/2}(Gv)\|^2 \, dt
	+ \epsilon \int_0^\tau \|D^{5/2}v\|^2 \, dt
		\leq \frac12\|v_0\|^2 + \int_0^T \|v\|_{3/2}\|F\|_{-3/2} \, dt.
\end{equation*}
and so
\begin{align} \label{M0}
&\|v\|_{L^\infty(0,T;H^0(\T))}^2
	+ \int_0^T \|D^{3/2}(Gv)\|^2 \, dt
	+ \epsilon \int_0^T \|D^{5/2}v\|^2 \, dt \notag\\
	&\qquad\qquad\qquad\qquad\qquad
		\leq \|v_0\|^2 + 2\int_0^T \|v\|_{3/2}\|F\|_{-3/2} \, dt.
\end{align}

We next apply a propagation of regularity argument to
account for the extra $3/2$-derivatives above.
We begin by introducing a function $b \in C_0^\infty(\omega)$,
$\omega = \{x\in\T : g(x)>0\}$, which forms a partition of unity.
Picking $t_0 \in (0,T)$,
\begin{equation} \label{D32Va}
\int_{t_0}^T \|D^{3/2}v\|^2 \, dt
	= \sum_k \int_{t_0}^T
			\left( b^2(x-x_k)D^{3/2}v, D^{3/2}v \right) \, dt.
\end{equation}
Notice for each $k$, there exists a primitive $\phi \in C^\infty(\T)$
which satisfies
\begin{equation} \label{PRIM}
b^2(x) - b^2(x-x_k) = \partial_x\phi(x).
\end{equation}
As each of the $k$ terms are estimated similarly
inserting \eqref{PRIM} yields
\begin{align} \label{D32Vb}
\int_{t_0}^T \|D^{3/2}v\|^2 \, dt
	&\lesssim
		\left|\int_{t_0}^T \left( b^2D^{3/2}v, D^{3/2}v \right) \, dt\right| \notag\\
	&\qquad\qquad
		+ 	\left|\int_{t_0}^T
				\left( \partial_x\phi D^{3/2}v, D^{3/2}v \right) \, dt\right| \notag\\
	&=: M_1 + M_2.
\end{align}
Following \cite{MR3335395}, observe that by definition
$b=g\tilde{b}$ for some $\tilde{b} \in C_0^\infty(\omega)$,
so that applying the commutator estimate \eqref{COMM} and interpolating
\begin{align} \label{M1a}
M_1
	&= \int_{t_0}^T \|bD^{3/2}v\|^2 \, dt \notag\\
	&\leq 2\int_{t_0}^T \|\tilde{b}D^{3/2}(gv)\|^2
				+ \|\tilde{b}[D^{3/2};g]v\|^2 \, dt \notag\\
	&\leq c\int_{t_0}^T \|D^{3/2}(gv)\|^2 \, dt
				+ c(\delta) \int_{t_0}^T \|v\|^2 \, dt
				+ \delta \int_{t_0}^T \|D^{3/2}v\|^2 \, dt.
\end{align}
Using the definition \eqref{CTRL} of $G$ produces
\begin{align} \label{M1b}
&\int_{t_0}^T \|D^{3/2}(gv)\|^2 \, dt \notag\\
	&\qquad\qquad
		\leq \int_{t_0}^T \|D^{3/2}(Gv)\|^2 \, dt
				+ \int_{t_0}^T \left|\int g(y)v(y,t)\, dy\right|^2
						\, \|D^{3/2}g\|^2 \, dt \notag\\
	&\qquad\qquad
		\leq \int_{t_0}^T \|D^{3/2}(Gv)\|^2 \, dt + c\int_{t_0}^T \|v\|^2 \, dt.
\end{align}
Combining \eqref{M1a} and \eqref{M1b}, then applying
\eqref{M0} we have
\begin{align} \label{M1}
M_1
	&\leq c\int_{t_0}^T \|D^{3/2}(Gv)\|^2 \, dt
				+ c(\delta) \int_{t_0}^T \|v\|^2 \, dt
				+ \delta \int_{t_0}^T \|D^{3/2}v\|^2 \, dt \notag\\
	&\leq c\|v_0\|^2
				 + c\int_0^T \|v\|_{3/2}\|F\|_{-3/2} \, dt \notag\\
	&\qquad\qquad
				+ c(\delta) \int_{t_0}^T \|v\|^2 \, dt
				+ \delta \int_{t_0}^T \|D^{3/2}v\|^2 \, dt
\end{align}
for any $\delta>0$ and with $c,c(\delta)$ independent of $\epsilon$ and $T$.

Because $v$ has mean value zero, $D=\mathcal{H}\partial_x$ and
\begin{equation} \label{M2a}
M_2 = \left| \int_{t_0}^T \int \partial_x\phi (\partial_x^2D^{-1/2}v)^2 \, dxdt \right|.
\end{equation}
Taking $w=D^{-1/2}v$ in \eqref{WGHT}, integrating in time
and applying the Sobolev embedding
\begin{align} \label{M2b}
&M_2
	\leq c(\|v(t_0)\|^2 + \|v(T)\|^2) \notag\\
	&\qquad\qquad
			+ \left|\int_{t_0}^T\int D^{3/2}(Gw)D^{3/2}G(\phi w) \, dxdt \right|
 			+ \left|\int_{t_0}^T\int wEw \, \phi \, dxdt \right| \notag\\
	&\qquad\qquad
			+ c\,\epsilon\, \int_{t_0}^T \|v\|^2 + \|D^{5/2}v\|^2 \, dt \notag\\
	&\qquad\qquad
			+ c(\delta) \int_{t_0}^T \|v\|^2 \, dt
			+ \delta \int_{t_0}^T \|D^{3/2}v\|^2 \, dt
			+ c\int_{t_0}^T \|v\|_{3/2}\|F\|_{-3/2} \, dt.
\end{align}
Assuming $0<\epsilon<1$, then applying \eqref{M0} produces
\begin{align} \label{M2c}
M_2
	&\leq c\|v_0\|^2
				 + c\int_0^T \|v\|_{3/2}\|F\|_{-3/2} \, dt \notag\\
	&\qquad\qquad
				+ c(\delta) \int_{t_0}^T \|v\|^2 \, dt
				+ \delta \int_{t_0}^T \|D^{3/2}v\|^2 \, dt
				+ M_{21} + M_{22}
\end{align}
where $M_{21}$ and $M_{22}$ are subsequently defined and estimated.
First note,
\begin{align} \label{M21a}
M_{21}
	&:= \left|\int_{t_0}^T\int D^{3/2}(Gw)D^{3/2}G(\phi w) \, dxdt \right| \notag\\
	&\leq \frac12\int_{t_0}^T \|D^{3/2}(Gw)\|^2 + \|D^{3/2}G(\phi w)\|^2 \, dt \notag\\
	&\leq c\int_{t_0}^T \|D^{3/2}(Gw)\|^2 \, dt + c\int_{t_0}^T \|v\|^2 \, dt
\end{align}
using identity \eqref{CTRL1}, the commutator estimate \eqref{COMM}
and the Sobolev embedding.
As $G$ is bounded on $H_0^s(\T)$, applying the commutator estimate \eqref{CTRL3} yields
\begin{align} \label{M21b}
\int_{t_0}^T \|D^{3/2}(Gw)\|^2 \, dt
	&\leq 2\int_{t_0}^T \|GD^{3/2}w\|^2 + \|[D^{3/2};G]w\|^2 \, dt \notag\\
	&\leq c(\delta) \int_{t_0}^T \|v\|^2 \, dt
				+ \delta \int_{t_0}^T \|D^{3/2}v\|^2 \, dt,
\end{align}
and so
\begin{equation} \label{M21}
M_{21}	\leq c(\delta) \int_{t_0}^T \|v\|^2 \, dt
				+ \delta \int_{t_0}^T \|D^{3/2}v\|^2 \, dt.
\end{equation}
Recalling the definition \eqref{EROP} of $E$
\begin{align} \label{M22a}
M_{22}
	&:= \left|\int_{t_0}^T\int wEw \, \phi \, dxdt \right| \notag\\
	&\leq \left|\int_{t_0}^T (\phi\, w,GD^3[D^{-1/2};G]D^{1/2}w) \, dt \right| \notag\\
	&\qquad\qquad
			+ \left|\int_{t_0}^T (\phi\, w,[D^{-1/2};G]D^3GD^{1/2}w) \, dt \right| \notag\\
	&=: M_{221} + M_{222}.
\end{align}
Using the commutator estimate \eqref{CTRL3}, \eqref{M21a} and \eqref{M21b} yields
\begin{align} \label{M221}
M_{221}
	&= \left|\int_{t_0}^T (D^{3/2}G(\phi w),D^{3/2}[D^{-1/2};G]v) \, dt \right| \notag\\
	&\leq \int_{t_0}^T \|D^{3/2}G(\phi w)\| \, \|D^{3/2}[D^{-1/2};G]v\| \, dt \notag\\
	&\leq c(\delta) \int_{t_0}^T \|v\|^2 \, dt
				+ \delta \int_{t_0}^T \|D^{3/2}v\|^2 \, dt.
\end{align}
Using the identity \eqref{CTRL2} with $r=3$ and $f=Gv$ produces
\begin{align} \label{M222}
M_{222}
	&\leq \left|\int_{t_0}^T (\phi\, w,[D^{-1/2};g]D^3Gv) \, dt \right|
				+ c\int_{t_0}^T \|v\|^2 \, dt \notag\\
	&\leq \int_{t_0}^T \|D^{3/2}[D^{-1/2};g]\phi w\| \, \|D^{3/2}(Gv)\| \, dt
				+ c\int_{t_0}^T \|v\|^2 \, dt \notag\\
	&\leq c\int_{t_0}^T \|v\|^2 \, dt + \frac12\int_{t_0}^T \|D^{3/2}(Gv)\|^2 \, dt,
\end{align}
after utilizing the commutator estimate \eqref{COMM}.
Collecting \eqref{M1}-\eqref{M222}, then applying \eqref{M0} we have
\begin{align} \label{D32Vc}
\int_{t_0}^T \|D^{3/2}v\|^2 \, dt
	&\leq c\int_{t_0}^T \|D^{3/2}(Gv)\|^2 \, dt
				+ c\int_0^T \|v\|_{3/2}\|F\|_{-3/2} \, dt \notag\\
	&\qquad\qquad
				+ c(\delta) \int_{t_0}^T \|v\|^2 \, dt
				+ \delta \int_{t_0}^T \|D^{3/2}v\|^2 \, dt \notag\\
	&\leq c\|v_0\|^2
				 + c\int_0^T \|v\|_{3/2}\|F\|_{-3/2} \, dt \notag\\
	&\qquad\qquad
				+ c(\delta) \int_{t_0}^T \|v\|^2 \, dt
				+ \delta \int_{t_0}^T \|D^{3/2}v\|^2 \, dt
\end{align}
for any $\delta>0$ and with $c,c(\delta)$ independent of $\epsilon$ and $T$.
Thus fixing $\delta=1/2$ and taking the limit $t_0\rightarrow0$ produces
\begin{equation} \label{D32V}
\int_0^T \|D^{3/2}v\|^2 \, dt
	\leq c(T)\left(\|v_0\|^2 + \int_0^T \|v\|_{3/2}\|F\|_{-3/2} \, dt\right),
\end{equation}
after again applying \eqref{M0}, for some $c(T)>0$ nondecreasing in $T$
and independent of $0<\epsilon<1$.
Adding this to \eqref{M0},
\begin{align} \label{UNIFORM0}
&\|v\|_{L^\infty(0,T;H^0(\T))}^2 + \int_0^T \|D^{3/2}v\|^2 \, dt \notag\\
	&\qquad\qquad
		\leq c(T)\left(\|v_0\|^2 + \int_0^T \|v\|_{3/2}\|F\|_{-3/2} \, dt\right) \notag\\
	&\qquad\qquad
		\leq c(T)\left(\|v_0\|^2 + \|F\|_{L^2(0,T;H^{-3/2}(\T))}^2 \right)
				+ \frac12 \int_0^T \|D^{3/2}v\|^2 \, dt.
\end{align}
The result holds for $v_0 \in H_0^0(\T)$ and $F \in L^2(0,T;H^{-3/2}(\T))$
by density.

\noindent
{\it (General case $s\neq0$.)} \\
Again assume $v_0 \in H_0^{s+5}(\T)$ and
$F \in C([0,T];H_0^{s+5}(\T))$ to justify what follows.
Applying $D^s$ to the equation \eqref{I:LINEAR} and
scaling by $w=D^sv$ yields
\begin{align} \label{N0a}
&\|w\|_{L^\infty(0,T;H^0(\T))}^2
	+ \int_0^T \|D^{3/2}(Gw)\|^2 \, dt
	+ \epsilon \int_0^T \|D^{5/2}w\|^2 \, dt \notag\\
	&\qquad\qquad
		\leq \|w_0\|^2
				+ 2\int_0^T \|w\|_{3/2}\|F\|_{s-3/2} \, dt
				+ 2\left|\int_0^T (w,Ew) \, dt\right|.
\end{align}
Recalling the definition \eqref{EROP} of $E$ we write
\begin{align} \label{N1}
N_1
	&= \left|\int_0^T\int wEw \, dxdt \right| \notag\\
	&\leq \left|\int_0^T (w,GD^3[D^s;G]D^{-s}w) \, dt \right| \notag\\
	&\qquad\qquad
			+ \left|\int_0^T (w,[D^s;G]D^3GD^{-s}w) \, dt \right| \notag\\
	&=: N_{11} + N_{12}.
\end{align}
Proceeding as in \eqref{M221},
\begin{align} \label{N11}
N_{11}
	&= \left|\int_0^T (D^{3/2}(Gw),D^{3/2}[D^s;G]D^{-s}w) \, dt \right| \notag\\
	&\leq \int_0^T \|D^{3/2}(Gw)\| \, \|D^{3/2}[D^s;G]D^{-s}w\| \, dt \notag\\
	&\leq \frac12\int_0^T \|D^{3/2}(Gw)\|^2 \, dt + c\int_0^T \|D^{1/2}w\|^2 \, dt.
\end{align}
and as in \eqref{M222}
\begin{align} \label{N12}
N_{12}
	&\leq \left|\int_0^T (w,[D^s;g]D^3GD^{-s}w) \, dt \right|
				+ c\int_0^T \|w\|^2 \, dt \notag\\
	&\leq \int_0^T \|D^{2-s}[D^{s};g]w\| \, \|D^{1+s}GD^{-s}w\|^2 \, dt
				+ c\int_0^T \|w\|^2 \, dt \notag\\
	&\leq c\int_0^T \|Dw\|^2 \, dt.
\end{align}
Collecting \eqref{N1}-\eqref{N12} we have
\begin{align} \label{N0}
&\|w\|_{L^\infty(0,T);H^0(\T)}^2
	+ \frac12\int_0^T \|D^{3/2}(Gw)\|^2 \, dt
	+ \epsilon \int_0^T \|D^{5/2}w\|^2 \, dt \notag\\
	&\qquad\qquad
		\leq \|w_0\|^2
				+ c\int_0^T \|Dw\|^2 \, dt
				+ 2\int_0^T \|w\|_{3/2}\|F\|_{s-3/2} \, dt.
\end{align}

The same propagation of regularity argument as in the
$s=0$ case reveals
\begin{equation} \label{D32w}
\int_0^T \|D^{3/2}w\|^2 \, dt
	\leq c(T)\left(\|w_0\|^2 + \int_0^T \|w\|_{3/2}\|F\|_{s-3/2} \, dt\right),
\end{equation}
for some $c(T)>0$ nondecreasing in $T$ and independent of $0<\epsilon<1$.
Combining \eqref{N0} and \eqref{D32w}, a density argument
shows the estimate \eqref{UNIFORM} holds for any $s>0$.
\end{proof}

Solutions to the IVP \eqref{I:LINEAR} are obtained via semigroup theory
by writing $L_\epsilon=A+B$ where
\begin{equation*}
A := \epsilon D^5 - \partial_x^5
	\quad\text{and}\quad
		B := GD^3G.
\end{equation*}
A perturbation argument shows that $L_\epsilon$ is sectorial.

\begin{proposition} \label{P:SEC}
Let $\epsilon>0$.
The operator $L_\epsilon$ is sectorial in $H_0^0(\T)$ and thus
$-L_\epsilon$ generates an analytic semigroup denoted $\{\mathcal{S}_\e(t)\}_{t\geq0}$.
Moreover, this semigroup acts on $H_0^s(\T)$ for any $s\geq0$.
\end{proposition}
\begin{proof}
The operator $A$ has domain $\mathcal{D}(A) = H_0^5(\T) \subset H_0^0(\T)$.
Fixing $\theta \in (\arctan\epsilon^{-1},\pi)$, it is clear that the sector
$$
S_\theta = \{ \lambda : \theta < |\arg\lambda| \leq \pi, \lambda \neq 0 \}
$$
lies in its resolvent.
Moreover, there exists $C>0$ so that for any $\lambda \in S_\theta$,
$$
\|(A-\lambda)^{-1}\| \leq \sup_{k\neq0} |(\epsilon-i)k^5-\lambda| \leq \frac{C}{|\lambda|}.
$$
Thus $A$ is sectorial in $H_0^0(\T)$ \cite[Definition 1.3.1]{MR610244}.

Observe that $\sigma(A) = \{ (\epsilon - i)k^5 : k\in\Z^* \}$ so that $\text{Re} \; \sigma(A) \geq \epsilon$.
Therefore, $A^{-\omega}$ is defined for all $\omega>0$ and, in particular,
\begin{align*}
\|BA^{-3/5}f\|^2
	&\leq c \|A^{-3/5}f\|_{H_0^3(\T)}^2 \\
	&\leq c \sum_{k\neq0} (1+k^2)^3 |(\epsilon - i)k^5|^{-6/5} |\hat{f}_k|^2 \\
	&\leq c \|f\|^2.
\end{align*}
It follows that $L_\epsilon=A+B$ is a sectorial operator
on $H_0^0(\T)$ \cite[Corollary 1.4.5]{MR610244}.
Therefore $-L_\epsilon$ generates an analytic semigroup
$\{\mathcal{S}_\epsilon(t)\}_{t\geq0}$ on $H_0^0(\T)$
\cite[Theorem 1.3.4]{MR610244}.
Using \cite[Theorem 1.4.8]{MR610244}, we can compute explicitly
$\mathcal{D}((A+B+\lambda)^\beta)=\mathcal{D}(A^\beta)=H_0^{5\beta}(\T)$
for all $\beta\geq0$ and $\lambda>0$ large enough,
hence for all $t>0$ and $s\geq0$,
$$
\mathcal{S}_\e(t)H_0^s(\T) \subset H_0^s(\T),
$$
as desired.
\end{proof}

The following unique continuation principle leads to exponential
stability and exact control results for solutions to
IVP \eqref{I:LINEAR}.
\begin{proposition} \label{P:UCP}
Let $c \in L^2(0,T)$ and $v \in L^2(0,T;H_0^0(\T))$
be such that
\begin{equation} \label{UCP}
\begin{cases}
	\partial_tv + (\epsilon D^5 - \partial_x^5)v = 0,
		& \text{in $\T\times(0,T)$} \\
	v(x,t) = c(t)
		& \text{for a.e. $(x,t)\in(a,b)\times(0,T)$}.
\end{cases}
\end{equation}
for some numbers $T>0$ and $0 \leq a < b \leq 2\pi$.
Then $v\equiv0$ for a.e. $(x,t)\in\T\times(0,T)$.
\end{proposition}

\begin{proof}
\noindent
{\it (Case $\epsilon=0$.)} \\
By assumption, $\partial_x^2v=0$ a.e. in $(a,b)\times(0,T)$
and so a propagation of regularity argument as in
Proposition \ref{P:UNI} implies $v \in L^2(0,T;H_0^2(\T))$.
Thus for every $\delta>0$, there exists $0<t<\delta$ such
that $v(t) \in H_0^2(\T)$.

In fact $\partial_x^kv=0$ a.e. in $(a,b)\times(0,T)$
for every $k\in\Z^+$.
Repeating the above argument and using the equation we
conclude that $v \in C^\infty((0,T)\times\T)$.
The unique continuation property now follows from the result
in \cite{MR871574}.

\noindent
{\it (Case $\epsilon>0$.)} \\
From \eqref{UCP} it follows that for a.e.
$(x,t)\in (a,b)\times(0,T)$,
$$\partial_tv=(1+\epsilon \mathcal H)\partial_x^5v = c'(t).$$
Moreover, for a.e. $t\in(0,T)$
\begin{align*}
    \partial_x^6v(\cdot,t) &\in H^{-6}(\T),\\
	\partial_x^6v(\cdot,t)&=0,\\
	\mathcal H\partial_x^6v(\cdot,t) &=0,
\end{align*}
since $\epsilon>0$.
Picking such a $t\in(0,T)$ and setting
$w(\cdot)=\partial_x^6v(\cdot,t)$, write
$$w=\sum\limits_{k\in\Z}\widehat{w}_ke^{ikx}$$
where the convergence occurs in $H^{-6}(\T)$.
Now
$$0=iw-\mathcal Hw=2i\sum\limits_{k>0}\widehat{w}_ke^{ikx}.$$
As $w$ is real, we use the following result.
\begin{lemma} \cite[Lemma 2.9]{MR3335395}
Let $s\in\R$ and $w(x)=\sum_{k\geq0}\widehat{w}_ke^{ikx}\in H^s(\T)$
and $w=0$ a.e. $(a,b)$.
Then $w\equiv 0$ in $\T$. 
\end{lemma}
Thus $\partial_x^6v\equiv 0$ in $\T$ which implies that
$v(x,t)=c(t)$ a.e. $\T\times(0,T)$.
Furthermore, $(1+\epsilon \mathcal H)\partial_x^6v=c_t(t)=0$ and
since $v$ has mean value zero we have shown the desired outcome that
$v\equiv 0$ in $\T\times(0,T)$.
\end{proof}

The above propositions imply an observability inequality leading
to the following result.

\begin{proposition} \label{P:STAB}
Let $0<\epsilon<1$, $s\geq0$.
There exists constants $C,\lambda>0$ independent of $\epsilon$
such that
\begin{equation*}
    \|\mathcal{S}_\epsilon(t)v_0\|_s \leq Ce^{-\lambda t}\|v_0\|_s
\end{equation*}
for all $v_0 \in H_0^s(\T)$.
\end{proposition}

\begin{proof}
\noindent
{\it (Case $s=0$.)} \\
Setting $F\equiv0$ in \eqref{I:LINEAR2} and scaling by $v$ yields 
for any $T>0$
\begin{equation} \label{ENERGY}
\frac{1}{2}\|v(T)\|^2 + \epsilon \int_0^T \|D^{5/2}v\|^2 \, d\tau
    + \int_0^T\|D^{3/2}(Gv)\|^2 \, d\tau
	    = \frac{1}{2}\|v_0\|^2,
\end{equation}
and so stability follows from the observability inequality
\begin{equation}\label{OBS}
\|v_0\|^2 
	\leq c \left( \epsilon \int_0^T \|D^{5/2}v\|^2 \, dt
	        + \int_0^T \|D^{3/2} (Gv)(t)\|^2 \, dt \right).
\end{equation}
For the sake of a contradiction, suppose that \eqref{OBS} fails.
Then there is a sequence $\{v_0^n\}_{n=1}^{\infty}\subset H_0^0(\T)$, (up to scaling)
such that
\begin{equation} \label{OBS2}
1
	= \|v_0^n\|^2
	> n \left( \epsilon \int_0^T \|D^{5/2}v\|^2 \, dt
	        + \int_0^T \|D^{3/2} (Gv^n)(t)\|^2 \, dt \right),
\end{equation}
with $v^n$ denoting the solution to \eqref{I:LINEAR}
corresponding to data $v_0^n$.
For any $\delta>0$, denote $\gamma = -\frac72-\delta$.
Applying the Sobolev embedding,
\begin{equation*}
\|(\epsilon D^5-\partial_x^5)v^n\|_{L^2(0,T;H^{\gamma}(\T))}
	\leq c\|v^n\|_{L^2(0,T;H^{3/2}(\T))},
\end{equation*}
which is uniformly bounded by the estimates \eqref{UNIFORM}.
Using commutator estimates, the Sobolev embedding and
the fact that $G$ is bounded on $H_0^0(\T)$,
\begin{align*}
\|D^{\gamma}G(D^3Gv^n)\|
	&\leq \|G(D^{3+\gamma}Gv^n)\| + \|[D^{\gamma},G]D^3Gv^n\| \\
	&\leq c\left(\|Gv^n\|_{3+\gamma} + \|Gv^n\|_{2+\gamma}\right) \\
	&\leq c\|Gv^n\|_{3/2}.
\end{align*}
Consequently,
\begin{equation*}
\|GD^3Gv^n\|^2_{L^2(0,T;H^{\gamma}(\T))}
	\leq c\int_0^T \|D^{3/2} (Gv^n)\|^2dt \leq C
\end{equation*}
using \eqref{ENERGY}.
Combining these estimates and using the equation produces
\begin{equation*}
\|v_t^n\|_{L^2(0,T;H^{\gamma}(\T))} \leq C
\end{equation*}
for some $C>0$ independent of $n$.
Note that $\{v_t^n\}$ is bounded in $L^2(0,T;H^{\gamma}(\T))$ and,
from \eqref{UNIFORM}, the sequence $\{v^n\}$ is bounded in
$L^2(0,T;H^{3/2}(\T))$.
Applying the Banach-Alaoglu theorem and the Aubin-Lions lemma,
we obtain a subsequence with the following properties:
$$\begin{array}{llc}
v^n\rightarrow v  &  \textrm{in } L^2(0,T; H^{\beta}(\T)) & \forall \beta<3/2   \\
v^n\rightarrow v  &  \textrm{in } L^2(0,T; H^{3/2}(\T)) \textrm{ weak} &   \\
v^n\rightarrow v  &  \textrm{in } L^{\infty}(0,T; L^2(\T)) \textrm{ weak}\ast, &   
\end{array}$$
where $v\in L^2(0,T;H_0^{\beta}(\T))\cap L^{\infty}(0,T;L^2(\T))$.
In particular, taking $\beta=0$
\begin{equation*}
(v^n)^2\rightarrow v^2\qquad \textrm{in }L^1(\T\times(0,T)).
\end{equation*}
Letting $n\rightarrow \infty$ in \eqref{OBS2} we have that
\begin{equation*}
\int_0^T \|D^{3/2}(Gv)\|^2 \, dt = 0.
\end{equation*}
Hence $Gv=0$ a.e. $\T\times(0,T)$ and using \eqref{CTRL} we may write
\begin{equation*}
v(x,t) = \int_\T g(y) v(y,t) \, dy := c(t)
	\qquad \text{for all $(x,t)\in\omega\times(0,T),$}
\end{equation*}
where $\omega=\{x\in\T:g(x)>0\}$ and $c \in L^\infty(0,T)$.
Thus $v$ satisfies the hypothesis of Proposition \ref{P:UCP}, implying
that $v\equiv0$ and contradicting the fact that $\|v(0)\|=\|v_0^n\|=1$.

\noindent
{\it (Case $s=5$.)} \\
Assume $v_0 \in H^5(\T)$ and denote $v(t)=\mathcal{S}(t)v_0$.
Let $w=\partial_tv$, which solves
\begin{equation*}
\begin{cases}
	\partial_tw + (\epsilon D^5 - \partial_x^5)w + GD^3Gw = 0,
		\qquad x\in\T, t\geq0, \\
	w(x,0) = w_0(x) := (\epsilon D^5 - \partial_x^5 + GD^3G)v_0,
\end{cases}
\end{equation*}
and so by the $s=0$ case
\begin{equation*}
\|w(t)\| = \|\mathcal{S}(t)w_0\| \leq Ce^{-\lambda t}\|w_0\|.
\end{equation*}
Using the equation \eqref{I:LINEAR} and the previous estimate
\begin{align*}
\|D^5v\|
    &\leq C\|(\epsilon D^5 - \partial_x^5)v\| \\
	&\leq C\|\partial_tv + GD^3Gv\| \\
	&\leq \left(Ce^{-\lambda t}\|w_0\|
					+ c(\delta)\|v\|\right) + \delta\|D^5v\|,
\end{align*}
for any $\delta>0$.
Choosing $\delta>0$ small enough,
\begin{align*}
\|D^5v\|
	&\leq (1-\delta)^{-1}\left(Ce^{-\lambda t}\|v_0\|_5
					+ c(\delta)e^{-\lambda t}\|v_0\|\right) \notag\\
	&\leq Ce^{-\lambda t}\|v_0\|_5.
\end{align*}
where $\lambda>0$ is as in the case $s=0$.
Interpolating produces the desired result for $0 \leq s \leq 5$,
with the case of $s>5$ following by induction.
Thus the constant appearing above will be nondecreasing in $s$.
\end{proof}

We now establish solutions as $\epsilon\searrow0$ using a
Bona-Smith argument \cite{MR0385355}.
The resulting homogeneous solutions to \eqref{I:LINEAR2} will be denoted
$\mathcal{S}(t)v_0$.

\begin{proposition} \label{P:BSA}
Fix $s\in\R$ and $T>0$.
Let $v_0 \in H_0^s(\T)$ and $F \in L^2(0,T;H_0^{s-3/2}(\T))$.
Then there exists a unique solution $v \in Z_{s,T}$
to the IVP
\begin{equation} \label{I:LINEAR2}
\begin{cases}
	\partial_tv - \partial_x^5v + GD^3Gv = F,
		\qquad x\in\T, t\geq0, \\
	v(x,0) = v_0(x)
\end{cases}
\end{equation}
satisfying
\begin{equation} \label{ZST}
\|v\|_{s,T} \leq c(s,T)\left( \|v_0\|_s + \|F\|_{L^2(0,T;H^{s-3/2}(\T))} \right)
\end{equation}
with $c(s,T)$ nondecreasing in $T$.
Moreover, there exists constants $C,\lambda>0$, $C=C(s)$, such that
\begin{equation} \label{LINEX}
    \|\mathcal{S}(t)v_0\|_s \leq Ce^{-\lambda t}\|v_0\|_s
\end{equation}
for all $v_0 \in H_0^s(\T)$, $s\geq0$.
\end{proposition}

\begin{proof}
We follow the argument of Bona and Smith to establish existence
of solutions to the IVP \eqref{I:LINEAR2}.
Define the regularization
\begin{equation} \label{REG1}
\widehat{v_0^\epsilon}(k) = \exp(-\epsilon^{1/10}k^2)\hat{v}_0(k)
\end{equation}
and observe that $v_0^\epsilon \in H^\infty(\T)$ and for
$\epsilon$ sufficiently small
\begin{equation} \label{REG2}
\epsilon^{\gamma/10}\|v_0^\epsilon\|_{s+\gamma}
	\leq c(\gamma) \|v_0\|_s
\end{equation}
for any $\gamma>0$.
Let $\{\epsilon_n\}\subset (0,1)$ be a monotonic sequence
satisfying $\epsilon^n\searrow0$ and denote
$v_0^n = v_0^{\epsilon^n}$.
Observe that $v_0^n \rightarrow v_0$ strongly in $H^s(\T)$.
Let $F^n \in C([0,T];H^\infty(\T))$ be a sequence converging
strongly to $F$ in $L^2(0,T;H^{s-3/2}(\T))$.
Let $v^n$ be the associated solution to the IVP
\begin{equation}
\partial_tv^n + L_{\epsilon^n}v^n = F^n,
	\qquad v^n(0)=v_0^n
\end{equation}
provided by Proposition \ref{P:UNI}.

We now demonstrate that the sequence $\{v^n\}$ is Cauchy
in $Z_{s,T}$ by considering
$$
\|v^n-v^m\|_{s,T}
$$
assuming $n \leq m$ (so that $0<\epsilon^m<\epsilon^n$).
The difference $w=v^n-v^m$ is a smooth solution to
\begin{equation}
\partial_tw + L_{\epsilon^m}w + (\epsilon^m-\epsilon^n)v^m = F^n-F^m,
	\qquad w(0)=v_0^n-v_0^m.
\end{equation}
Thus taking 
$$
F:=F^n-F^m-(\epsilon^n-\epsilon^m)D^5v^n
$$
in \eqref{UNIFORM} produces
\begin{align} \label{Q0a}
\|w\|_{s,T}
	&\leq c(s,T)\left(\|w_0\|_s^2 + \|F^n-F^m\|_{L^2(0,T;H^{s-3/2} (T))}\right) \notag\\
	&\qquad\qquad
			+ (\epsilon^n-\epsilon^m) \, c(s,T) \|D^{s+5}v^n\|_{L^2(0,T;H^{s-3/2}(T))}.
\end{align}
Applying \eqref{UNIFORM} to $v^n$
\begin{align} \label{Q0b}
&(\epsilon^n-\epsilon^m) \|D^{s+5}v^n\|_{L^2(0,T;H^{s-3/2}(T))} \\
	&\qquad\qquad
		\leq \epsilon^n \|v^n\|_{s+2,T} \notag\\
	&\qquad\qquad
		\leq (\epsilon^n)^{3/10}\,
				c(s,T)\left((\epsilon^n)^{7/20}\|v_0^n\|_{s+2}
				+(\epsilon^n)^{7/20}\|F^n\|_{L^2(0,T;H^{s+2} (T))}\right) \notag\\
	&\qquad\qquad
		\leq (\epsilon^n)^{3/10}\,c(s,T)\left(\|v_0\|_s+ \|F\|_{L^2(0,T;H^{s-3/2} (T))}\right),
\end{align}
where we utilized \eqref{REG2} with $\gamma=7/2$,
$\epsilon^n-\epsilon^m\leq\epsilon^n<1$ and
that $F^n \rightarrow F$ strongly.
Inserting \eqref{Q0b} into \eqref{Q0a} yields
\begin{align*}
\|w\|_{s,T}
	&\leq c(s,T)\left(\|w_0\|_s + \|F^n-F^m\|_{L^2(0,T;H^{s-3/2} (T))}\right) \notag\\
	&\qquad\qquad
			+ (\epsilon^n)^{3/10}\, c(s,T) \left(\|v_0\|+ \|F\|_{L^2(0,T;H^{s-3/2} (T))}\right).
\end{align*}
This proves that $\{v^n\}$ is Cauchy in $Z_{s,T}$, thus
$v^n \rightarrow v$ for some $v \in Z_{s,T}$.
Choosing $n$ large enough,
\begin{align*}
\|v\|_{s,T}
	&\leq \|v^n\|_{s,T} + \delta \\
	&\leq c(s,T)\left(\|v_0^n\|_s + \|F^n\|_{L^2(0,T;H^{s-3/2} (T))}\right) + \delta \\
	&\leq c(s,T)\left(\|v_0\|_s + \|F\|_{L^2(0,T;H^{s-3/2} (T))}\right)
				+ 3\delta,
\end{align*}
and so $v$ satisfies \eqref{ZST}.
Moreover, $v$ is a distributional solution of IVP \eqref{I:LINEAR2}
with $v(\cdot,t)\rightarrow v_0$ strongly in $H^s(\T)$
as $t\rightarrow0$.
Uniqueness and continuous dependence on the initial data follow
easily from \eqref{ZST}.
Finally, \eqref{LINEX} holds as the results of Proposition \ref{P:STAB}
are independent of $0<\epsilon<1$.
\end{proof}

\end{section}

\begin{section}{Exponential Stabilization}\label{S:3}

This section is concerned with local well-posedness and
stabilization of solutions to the following nonlinear equation
\begin{equation} \label{I:STBL}
\begin{cases}
	\partial_tu -\partial_x^5u + u\partial_x^3u = - GD^3Gu,
		\qquad x\in\T, t\geq0, \\
	u(x,0) = u_0(x).
\end{cases}
\end{equation}
The linear estimates given in Proposition \ref{P:BSA}, when
combined with the contraction principle yield local
well-posedness for small data in $H_0^s(\T)$ for $s>2$.

\begin{theorem} \label{T:LWP_REG}
Suppose $s>2$ and $T>0$.
Then there exists $\rho=\rho(s,T)>0$ such that for any
$u_0 \in H_0^s(\T)$ with $\|u_0\|_s\leq\rho$, the IVP
\eqref{I:STBL} admits a unique solution in the space $Z_{s,T}$.
\end{theorem}
\begin{proof}
We write \eqref{I:STBL} in the integral form
\begin{equation*}
u(t)
	= \mathcal{S}(t)u_0
			- \int_0^t \mathcal{S}(t-t')(u\partial_x^3u)(t') \, dt'
	=: \Gamma(u)
\end{equation*}
and show that $\Gamma$ defines a contraction on
$B = \{ v \in Z_{s,T} : \|v\|_{s,T} \leq R \}$ for
appropriate choices of $R>0$ and $\rho>0$.
Note that the restriction $s>2$ ensures that $H^{s-3/2}(\T)$
forms a Banach algebra.
The estimate \eqref{UNIFORM} yields
\begin{equation*}
\|\Gamma(u)\|_{s,T}
	\leq c(s,T)\left( \|u_0\|_s
				+ \|u\partial_x^3u\|_{L^2(0,T;H^{s-3/2}(\T))} \right).
\end{equation*}
Assuming $u \in Z_{s,T}$, then
\begin{align*}
\int_0^T \|u\partial_x^3u\|_{s-3/2}^2 \, dt
	&\leq c\int_0^T \left(\|u\|_{s-3/2}\|\partial_x^3u\|_{s-3/2}\right)^2 \, dt \\
	&\leq c\int_0^T \|u\|_s^2\|u\|_{s+3/2}^2 \, dt \\
	&\leq c\|u\|_{L^\infty(0,T;H^s(\T))}^2\|u\|_{L^2(0,T;H^{s+3/2}(T))}^2 \\
	&\leq c\|u\|_{s,T}^4.
\end{align*}
Therefore
\begin{equation*}
\|\Gamma(u)\|_{s,T} \leq C_0\|u_0\|_s + C_1\|u\|_{s,T}^2
\end{equation*}
for some $C_0,C_1>0$ (which depend on $T$ through estimate
\eqref{UNIFORM}).
Next, assuming $u,v \in Z_{s,T}$ and writing
\begin{equation*}
u\partial_x^3u - v\partial_x^3v
	= (u\partial_x^3u - v\partial_x^3u) + (v\partial_x^3u - v\partial_x^3v),
\end{equation*}
the same estimates as above reveal
\begin{equation*}
\|\Gamma(u)-\Gamma(v)\|_{s,T}
	\leq C_1\left(\|u\|_{s,T}+\|v\|_{s,T}\right)\|u-v\|_{s,T}.
\end{equation*}
Thus $\Gamma$ forms a contraction on $B$ provided
\begin{equation*}
C_0\|u_0\|_s + C_1R^2 < R
	\quad\text{and}\quad
		2C_1R<1.
\end{equation*}
It is sufficient to take
\begin{equation*}
R=(4C_1)^{-1}
	\quad\text{and}\quad
		\|u_0\|_s \leq \rho := R/2C_0.
\end{equation*}
\end{proof}

Following \cite{MR3335395}, the contraction principle is
also used to establish local exponential stability of the
solutions to the IVP \eqref{I:STBL}.
However, the estimates in Proposition \ref{P:BSA}
incorporate only the regularizing effects of the control
term and not any stabilization.
As a result, the $H^s$-estimates \eqref{ZST} possibly
grow in time.
This artifact is avoided by restricting \eqref{ZST}
to (at most) unit length time intervals through use of the spaces
\begin{equation*}
Z_{s,T}([n,n+1])
	:= C([n,n+1];H_0^s(\T)) \cap L^2(n,n+1;H_0^{s+3/2}(\T))
\end{equation*}
endowed with the norm
\begin{equation*}
\vertiii{u}_n
	:= \|u\|_{L^\infty(n,n+1;H^s(\T))}
			+ \|u\|_{L^2(n,n+1;H^{s+3/2}(\T))}.
\end{equation*}
Proposition \ref{P:BSA} leads to the following linear estimates.

\begin{proposition} \label{P_NLH}
Let $0 \leq s \leq 5$ and $v_0 \in H_0^s(\T)$.
Then for some $\lambda,\tilde{c}_0,\tilde{c}_1>0$ independent
of $s$ and $t$,
\begin{equation} \label{LIN_NH}
\vertiii{\mathcal{S}(t)u_0}_n
	\leq \tilde{c}_0 e^{-\lambda t}\|u_0\|_s
\end{equation}
and
\begin{align} \label{LIN_NI}
&\vertiii{\int_0^t \mathcal{S}(t-t')F(t') \, dt'}_n \notag\\
	&\qquad\qquad
	\leq \tilde{c}_1
				\left(\|F\|_{L^2(n,n+1;H^{s-3/2}(\T))}
						+ \sum_{k=1}^n e^{-\lambda(n-k)} \|F\|_{L^2(k-1,k;H^{s-3/2}(\T))}\right).
\end{align}
\end{proposition}
\begin{proof}
From \eqref{LINEX}
\begin{equation*}
\|\mathcal{S}(n)u_0\|_s \leq ce^{-\lambda n}\|u_0\|_s
\end{equation*}
so that, combined with \eqref{ZST},
\begin{equation*}
\vertiii{\mathcal{S}(t)u_0}_n \leq \tilde{c}_0e^{-\lambda n}\|u_0\|_s.
\end{equation*}

For the inhomogeneous estimates, write,
\begin{align*}
\vertiii{\int_0^t \mathcal{S}(t-t')F(t') \, dt'}_n
	&\leq \vertiii{\mathcal{S}(t-n)\int_0^n \mathcal{S}(n-t')F(t') \, dt'}_n \\
	&\qquad\qquad
				+ \vertiii{\int_n^t \mathcal{S}(t-t')F(t') \, dt'}_n \\
	&=: I_1 + I_2.
\end{align*}
Applying \eqref{ZST} and \eqref{LINEX} repeatedly over the time
intervals $[0,1],[1,2],\dots,[n-1,n],[n,t]$ yields,
since $n \leq t < n+1$,
\begin{align*}
I_1
	&\leq c\left\|\int_0^n\mathcal{S}(n-t')F(t') \, dt'\right\|_s \\
	&\leq c\sum_{k=1}^n
				\left\|\mathcal{S}(n-k)
						\int_{k-1}^k\mathcal{S}(k-t')F(t') \, dt'\right\|_s \\
	&\leq c\cdot\tilde{c}_0 \sum_{k=1}^n e^{-\lambda(n-k)}
				\left\|\int_{k-1}^k\mathcal{S}(k-t')F(t') \, dt'\right\|_s \\
	&\leq \tilde{c}_1 \sum_{k=1}^n e^{-\lambda(n-k)}
				\|F\|_{L^2(k-1,k,H^{s-3/2}(\T))},
\end{align*}
where we used \eqref{LIN_NH}.
Similarly,
\begin{equation*}
I_2 \leq \tilde{c}_1 \|F\|_{L^2(n,n+1,H^{s-3/2}(\T))},
\end{equation*}
completing the proof.
\end{proof}

We now prove Theorem \ref{T:STAB} under the assumption $[u_0]=0$.
\begin{proof}
We proceed via the contraction principle in the Banach space
\begin{equation*}
X
	:= \{ u \in C(\R^+;H_0^s(\T)) \cap L_{\text{loc}}^2(\R^+;H_0^{s+3/2}(\T)) : \|u\|_E < \infty \}
\end{equation*}
where
\begin{equation*}
\|u\|_X := \sup_{n\geq0} \left\{e^{n\lambda}\vertiii{u}_n\right\}.
\end{equation*}
Writing \eqref{E:KDV5C} in the integral form
\begin{equation*}
u(t)
	= \mathcal{S}(t)u_0
			- \int_0^t \mathcal{S}(t-t')(u\partial_x^3u)(t') \, dt'
	=: \Gamma(u),
\end{equation*}
then \eqref{LIN_NH} and \eqref{LIN_NI} produce
\begin{equation} \label{GAMU}
\|\Gamma(u)\|_X
	\leq \tilde{c}_0\|u_0\|_s
+ \sup_{n\geq0} \tilde{c}_1e^{n\lambda}
\begin{bcases}
  &\|u\partial_x^3u\|_{L^2(n,n+1;H^{s-3/2}(\T))}	 \\
  &+ \sum_{k=1}^n e^{-\lambda(n-k)}
			\|u\partial_x^3u\|_{L^2(k-1,k;H^{s-3/2}(\T))}
\end{bcases}.
\end{equation}
Using the algebra property of $H^{s-3/2}(\T)$,
\begin{align} \label{KEM}
\int_{k-1}^k \|u\partial_x^3u\|_{s-3/2}^2 \, dt
	&\leq c\int_{k-1}^k \left(\|u\|_{s-3/2}\|\partial_x^3u\|_{s-3/2}\right)^2 \, dt \notag\\
	&\leq c\int_{k-1}^k \|u\|_s^2\|u\|_{s+3/2}^2 \, dt \notag\\
	&\leq c\|u\|_{L^\infty(k-1,k;H^s(\T))}^2\|u\|_{L^2(k-1,k;H^{s+3/2}(T))}^2 \notag\\
	&\leq c\vertiii{u}_{k-1}^4.
\end{align}
Assuming $u \in E$, inserting \eqref{KEM} into
\eqref{GAMU}, we have
\begin{align*}
\|\Gamma(u)\|_X
	&\leq \tilde{c}_0 \|u_0\|_s
				+ \sup_{n\geq0}\tilde{c}_1e^{n\lambda}
					\left\{\vertiii{u}_n^2 + \sum_{k=1}^n e^{-\lambda(n-k)} \vertiii{u}_{k-1}^2\right\} \\
	&\leq \tilde{c}_0 \|u_0\|_s
				+ \tilde{c}_1\|u\|_X^2 + \tilde{c}_1\sup_{n\geq0}
					\left\{\sum_{k=1}^n e^{k \lambda} \vertiii{u}_{k-1}^2\right\} \\
	&\leq \tilde{c}_0 \|u_0\|_s
				+ \tilde{c}_1\|u\|_X^2
				+ \tilde{c}_1 e^{2\lambda}\|u\|_X^2
						\left\{\sup_{n\geq0} \, \sum_{k=1}^n e^{-k \lambda} \right\} \\
	&\leq \tilde{c}_0 \|u_0\|_s + \tilde{c}_1(\lambda)\|u\|_X^2.
\end{align*}
Similarly, assuming $u,v \in X$,
\begin{equation*}
\|\Gamma(u)-\Gamma(v)\|_X \leq \tilde{c}_1\left(\|u\|_X+\|v\|_X\right)\|u-v\|_X.
\end{equation*}
Thus $\Gamma$ forms a contraction on
$B = \{ v \in X : \|u\|_X \leq R\}$ provided
\begin{equation*}
\tilde{c}_0\|u_0\|_s + \tilde{c}_1R^2 < R
	\quad\text{and}\quad
		2\tilde{c}_1R<1.
\end{equation*}
It is sufficient to take
\begin{equation*}
R=(4\tilde{c}_1)^{-1}
	\quad\text{and}\quad
		\|u_0\|_s \leq \rho := R/2\tilde{c}_0.
\end{equation*}
\end{proof}

\begin{remark}
We now consider the previous two theorems applied to
\begin{equation} \label{E:KDV5GG}
\partial_tu - \partial_x^5u
	+ c_1u^2\partial_xu + c_2\partial_xu\partial_x^2u
	+ c_3u\partial_x^3u + GD^3Gu
		= 0.
\end{equation}
Observe that the nonlinearity $P(u)$ satisfies
\begin{equation*}
\int_\T P(u) \, dx = 0
\end{equation*}
and so solutions to \eqref{E:KDV5GG} preserve volume.
The  restriction $s>2$ arose from
utilizing the algebra property of $H^{s-3/2}(\T)$
in estimates of the form
\begin{equation} \label{EXT}
\int_0^T \|P(u)\|_{s-3/2}^2 \, dt.
\end{equation}
Hence Theorems \ref{T:STAB} and \ref{T:LWP_REG}
apply to equation \eqref{E:KDV5GG} with the same technique.
In fact, imposing $s>7/2$ and replacing the algebra
property with
\begin{equation*}
\|fg\|_s
	\leq c\left(\|f\|_{L_x^\infty}\|g\|_s
				+ \|f\|_s\|g\|_{L_x^\infty}\right)
\end{equation*}
the theorems extend to an even wider family of fifth order models.
\end{remark}

\begin{remark}
Next it is shown that Theorems \ref{T:STAB} and \ref{T:LWP_REG}
apply to a family of equations containing the KdV hierarchy.
Following Section \ref{S:3}, we see that for each $l\in\Z^+$,
$l\geq2$, the linear equation
\begin{equation*}
\partial_tv + (-1)^{l+1}\partial_x^{2l+1}v = - GD^{2l-1}Gv
\end{equation*}
possesses a unique solution in the space
\begin{equation*}
Z_{s,T}^l = C(0,T;H_0^s(\T)) \cap L^2(0,T;H_0^{s+l-1/2}(\T))
\end{equation*}
which decays exponentially in $H_0^s(\T)$ for $s\geq0$.
The algebra property holds for $H_0^{s-l+1/2}(\T)$
assuming $s>l$, and in this case
\begin{equation*}
\int_0^T \|u\partial_x^{2l-1}u\|_{s-l+1/2}^2 \, dt
	\leq c \|u\|_{L^\infty(0,T;H_0^s(\T))}^2
				\|u\|_{L^2(0,T;H_0^{s+l-1/2}(\T))}^2.
\end{equation*}
Therefore the equation
\begin{equation*}
\partial_tu + (-1)^{l+1}\partial_x^{2l+1}u + u\partial_x^{2l-1}u
    = -GD^{2l-1}Gu
\end{equation*}
is locally well-posed and exponentially stabilizable for
small data in $H_0^s(\T)$, $s>l$.
The nonlinearity $u\partial_x^{2l-1}u$ is the most
difficult to control in the following family studied by
Kenig and Pilod \cite{MR3513119} and Gr\"unrock \cite{MR2653659}:
\begin{equation*} \label{KDVH}
\partial_tu + (-1)^{l+1}\partial_x^{2l+1}u + GD^{2l-1}Gu
	= \sum_{k=2}^{l+1} N_{lk}(u),
\end{equation*}
where
\begin{equation*}
N_{lk}(u)
	= \sum_{|n|=2(l-k)+3}
			c_{l,k,n}\partial_x^{n_0}
				\prod_{i=1}^k \partial_x^{n_i}u,
\end{equation*}
with $|n|=\sum_{i=0}^k n_i$, $n_i\in\N$, for $i=0,\dots,k$
and $c_{l,k,n}\in\R$.
Further imposing $n_0>0$, this describes a family of
volume-preserving equations containing the KdV hierarchy
to which we have extended
Theorems \ref{T:STAB} and \ref{T:LWP_REG}.
\end{remark}

\end{section}

\begin{section}{Exact Controllability}\label{S:4}

This section is devoted to establishing exact
controllability of the equation
\begin{equation*}
\partial_tu - \partial_x^5u + u\partial_x^3u = Gh,
\end{equation*}
where $h$ is the control input.
Following \cite{MR3335395}, we
incorporate dissipation into the control input in order
to obtain a suitable smoothing effect.
We set
\begin{equation*}
h = -D^3Gu + D^{3/2}k,
\end{equation*}
viewing $k$ as the new control, and focus on the system
\begin{equation} \label{E:KDV5C}
\partial_tu - \partial_x^5u + GD^3Gu + u\partial_x^3u
	= GD^{3/2}k,
		\qquad u(0)=u_0.
\end{equation}
We first establish control of the associated linear system
in $H^s(\T)$ using the the Hilbert Uniqueness Method
(as in \cite{MR2486102}, \cite{MR3335395})
and then apply the contraction principle 
to obtain controllability of \eqref{E:KDV5C}.

\begin{proposition} \label{P:LEC}
Let $s\geq0$ and $T>0$.
Then for any $v_0,v_T \in H_0^s(\T)$, there exists
$k \in L^2(0,T;H_0^s(\T))$ such that
\begin{equation} \label{E:KDV5CL}
\partial_tv - \partial_x^5v + GD^3Gv = GD^{3/2}k
\end{equation}
admits a unique solution $v \in Z_{s,T}$ satisfying
$v(0)=v_0$ and $v(T)=v_T$.
\end{proposition}
\begin{proof}
\noindent
{\it (Case $s=0$.)} \\
Note that for
$v_0 \in H_0^s(\T)$ and $k \in L^2(0,T;H_0^s(\T))$
the solution to \eqref{E:KDV5CL} lies in $Z_{s,T}$.
We associate to this equation the adjoint system
\begin{equation} \label{E:KDV5CL*}
-\partial_tu + \partial_x^5u + GD^3Gu = 0,
	\qquad
	u(T) = u_T.
\end{equation}
Assuming $v_0, u_T \in H_0^5(\T)$
and $k \in L^2(0,T;H_0^{13/2}(\T))$ to justify the
computations, scaling \eqref{E:KDV5CL} by $u$ yields
\begin{equation} \label{DUAL}
\int_\T uv \, dx \Big|_0^T
	= \int_0^T \int_\T k D^{3/2}(Gu) \, dxdt,
\end{equation}
assuming $u_0=0$.
Hence duality implies that exact controllability
of \eqref{E:KDV5CL} follows from an observability
inequality
\begin{equation} \label{OBS3}
\|u_T\|^2 \leq C \int_0^T \|D^{3/2}(Gu)\|^2 \, dt
\end{equation}
for solutions to \eqref{E:KDV5CL*}.

Demonstrating \eqref{OBS3} requires a few properties
of these solutions. Note that scaling the adjoint
equation \eqref{E:KDV5CL*} by $tu$ yields
\begin{equation} \label{ENERGY2}
\frac{T}{2}\|u_T\|^2
	= \frac12 \int_0^T \|u(t)\|^2 \, dt
			+ \int_0^T t \|D^{3/2}(Gu)\|^2 \, dt.
\end{equation}
Moreover, a propagation of regularity argument
similar to Proposition \ref{P:UNI} (changing $t$
to $T-t$) shows that solutions
to \eqref{E:KDV5CL*} satisfy
\begin{equation} \label{PROP}
\int_0^T \|D^{3/2}u\|^2 \, dt \leq C(T) \|u_T\|^2.
\end{equation}

We now demonstrate \eqref{OBS3}.
Proceeding by contradiction, suppose there is a
sequence $\{u_T^n\}$ in $H_0^0(\T)$ such that
\begin{equation} \label{OBS3b}
1 = \|u_T^n\|^2 > n \int_0^T \|D^{3/2}(Gu^n)\|^2 \, dt.
\end{equation}
with $u^n$ denoting the solution to \eqref{E:KDV5CL*}
corresponding to data $u_T^n$.
Using the equation \eqref{E:KDV5CL*} and \eqref{PROP},
the sequence $\{u^n\}$ is seen to be bounded in
\begin{equation*}
L^2(0,T;H_0^{3/2}(\T)) \cap H^1(0,T;H_0^0(\T)).
\end{equation*}
The Aubin-Lions lemma implies the existence of a
subsequence (still denoted $u^n$) converging strongly
to a limit $u$ in $L^2(0,T;H_0^0(\T))$.

Next, we verify that $\{u_T^n\}$ is Cauchy in $H_0^0(\T)$.
Estimate \eqref{ENERGY2} applied to the difference of
two solutions yields
\begin{align*}
\|u_T^n-u_T^m\|
	&\leq \frac1T \int_0^T \|u^n-u^m\|^2 \, dt
				+ 2 \int_0^T \|D^{3/2}G(u^n-u^m)\|^2 \, dt \\
	&\leq \frac1T \int_0^T \|u^n-u^m\|^2 \, dt
				+ 4\left(\frac1n+\frac1m\right)
\end{align*}
after applying \eqref{OBS3b}.
Thus $u_T^n \rightarrow u_T$ strongly in $H_0^0(\T)$ and
it follows that the solution of \eqref{E:KDV5CL*}
associated to $u_T$ agrees with the limit $u$ of the
sequence $\{u^n\}$.
Letting $n\rightarrow \infty$ in \eqref{OBS3} we have that
\begin{equation*}
\int_0^T \|D^{3/2}(Gu)\|^2 \, dt = 0.
\end{equation*}
Hence $Gu=0$ a.e. $\T\times(0,T)$ and using \eqref{CTRL} we may write
\begin{equation*}
u(x,t) = \int_\T g(y) u(y,t) \, dy := c(t)
	\qquad \text{for all $(x,t)\in\omega\times(0,T),$}
\end{equation*}
where $\omega=\{ x\in\T:g(x)>0\}$ and $c \in L^\infty(0,T)$.
Thus $u$ satisfies the hypothesis of Proposition \ref{P:UCP},
implying that $u\equiv 0$ and contradicting the fact that
$\|u_T\|=\|u_T^n\|=1$.

\noindent
{\it (Case $s>0$.)} \\
As the IVP \eqref{E:KDV5CL*} is well-posed backwards
in time, we have
\begin{equation} \label{ZST_NEG}
\|u\|_{-s,T} \leq c(s,T)\|u_T\|_{-s}.
\end{equation}
As in \eqref{DUAL}, scaling \eqref{E:KDV5CL} by a solution
$u$ to \eqref{E:KDV5CL*} and supposing $v_0=0$, then
\begin{equation} \label{DUAL2}
\langle u_T, v(T) \rangle_{-s,s}
	= \int_0^T \langle D^{3/2}(Gu),k \rangle_{-s,s} \, dt,
\end{equation}
where $\langle \cdot,\cdot \rangle_{-s,s}$ denotes the
pairing $\langle \cdot,\cdot \rangle_{H_0^{-s}(\T),H_0^s(\T)}$.
Thus is suffices to prove the following observability inequality
\begin{equation} \label{OBS4}
\|u_T\|_{-s}^2 \leq C \int_0^T \|D^{3/2}(Gu)\|_{-s}^2 \, dt
\end{equation}
for solutions $u$ to \eqref{E:KDV5CL*}.

We first show that $w=D^{-s}u$ satisfies
\begin{equation} \label{OBS4_A}
\|w_T\|^2
	\leq C \left( \int_0^T \|D^{3/2}(Gw)\|^2 \, dt
				+ \int_0^T \|Ew\|_{-3/2}^2 \, dt \right)
\end{equation}
where
\begin{equation} \label{E:KDV5CLS*}
-\partial_tw + \partial_x^5w + GD^3Gw = Ew,
	\qquad w(T)= w_T := D^{-s}u_T
\end{equation}
and $E := D^{-s}[D^s;GD^3G]$.
To obtain a contradiction, suppose there is a sequence
$\{w_T^n\}$ in $H_0^0(\T)$ such that
\begin{equation} \label{OBS4_AF}
1 = \|w_T^n\|^2
	> n \left( \int_0^T \|D^{3/2}(Gw^n)\|^2 \, dt
				+ \int_0^T \|Ew^n\|_{-3/2}^2 \, dt \right),
\end{equation}
with $w^n$ denoting the solution to \eqref{E:KDV5CLS*}
corresponding to $w(T)=w_T^n$.
Then \eqref{ZST_NEG}, along with the equation satisfied
by $w^n$, implies that the sequence $\{w^n\}$ is bounded in
\begin{equation*}
L^2(0,T;H_0^{3/2}(\T)) \cap H^1(0,T;H_0^{-7/2}(\T)).
\end{equation*}
The Aubin-Lions lemma implies the existence of a 
subsequences (still denoted $w^n$) converging strongly
to a limit $w$ in $L^2(0,T;H_0^0(\T))$.
Next we verify that $\{w_T^n\}$ is Cauchy in $H_0^0(\T)$.
Scaling equation \eqref{E:KDV5CLS*} by $tw$ yields
\begin{equation} \label{ENERGY3}
\frac{T}{2}\|w_T\|^2
	= \frac12 \int_0^T \|w(t)^2\| \, dt
			+ \int_0^T t \|D^{3/2}(Gw)\|^2 \, dt
			+ \int_0^T\int_\T twEw \, dxdt,
\end{equation}
which also applies to the difference of two solutions
so that
\begin{align*}
\|w_T^n-w_T^m\|
	&\leq \frac1T \int_0^T \|w^n-w^m\|^2 \, dt
				+ 2 \int_0^T \|D^{3/2}G(w^n-w^m)\|^2 \, dt \\
	&\qquad\qquad
				+ 2 \int_0^T \|w^n-w^m\|_{3/2}\|\|E(w^n-w^m)\|_{-3/2} \, dt \\
	&\leq \frac1T \int_0^T \|w^n-w^m\|^2 \, dt \\
	&\qquad\qquad
				+ 4\left(\int_0^T \|D^{3/2}(Gw^n)\|^2 \, dt
						+ \int_0^T \|D^{3/2}(Gw^m)\|^2 \, dt \right) \\
	&\qquad\qquad
				+ \delta \int_0^T \|w^n-w^m\|_{3/2}^2 \, dt \\
	&\qquad\qquad
				+ c(\delta) \left( \int_0^T \|Ew^n\|_{-3/2}^2 \, dt
						+ \int_0^T \|Ew^m\|_{-3/2}^2 \, dt \right).
\end{align*}
Choosing $\delta>0$ small enough, the claim follows from
\eqref{ZST_NEG}, \eqref{OBS4_AF} and the strong convergence
of $w^n$ in $L^2(0,T;H_0^0(\T))$.
Thus $w_T^n \rightarrow w_T$ strongly in $H_0^0(\T)$ and
it follows that the solution of \eqref{E:KDV5CLS*}
associated to $w_T$ agrees with the limit $w$ of the
sequence $\{w^n\}$.
Letting $n\rightarrow \infty$ in \eqref{OBS4_AF} we have that
\begin{equation*}
\int_0^T \|D^{3/2}(Gw)\|^2 \, dt = 0.
\end{equation*}
Hence $Gw=0$ a.e. $\T\times(0,T)$ and an application of
Proposition \ref{P:UCP} implies that $w\equiv 0$, thus
contradicting the fact that $\|w_T\|=\|w_T^n\|=1$.	
Thus \eqref{OBS4_A} holds.

We now prove the following estimate of solutions $u$
to equation \eqref{E:KDV5CL*}
\begin{equation} \label{OBS4_B}
\|u_T\|_{-s}^2
	\leq C \left( \int_0^T \|D^{3/2}(Gu)\|_{-s}^2 \, dt,
				+ \|u_T\|_{-s-1}^2 \right),
\end{equation}
from which \eqref{OBS4} will follow.
To obtain a contradiction, suppose there is a sequence
$\{u_T^n\}$ in $H_0^{-s}(\T)$ such that
\begin{equation} \label{OBS4_BF}
1 = \|u_T^n\|_{-s}^2
	> n \left( \int_0^T \|D^{3/2}(Gu^n)\|_{-s}^2 \, dt
				+ \|u_T^n\|_{-s-1}^2 \right),
\end{equation}
with $u^n$ denoting the solution to \eqref{E:KDV5CL*}
corresponding to $u(T)=u_T^n$.
This implies $u_T^n\rightarrow0$ strongly in
$H_0^{-s-1}(\T)$ and so $u^n\rightarrow0$ in $Z_{-s-1,T}$.
Then
\begin{align*}
\int_0^T \|GD^{-s}u^n\|_{3/2}^2 \, dt
	&\leq c\left(\int_0^T \|D^{3/2}(Gu^n)\|_{-s}^2 \, dt \right. \\
	&\qquad\qquad
		\left. + \int_0^T \|D^{3/2}[D^{-s};G]u^n\|^2 \, dt,
				\vphantom{\int_0^T}\right)
\end{align*}
where the first term on the right-hand side tends towards zero
by \eqref{OBS4_BF}.
Applying commutator estimate \eqref{CTRL3},
\begin{align*}
\int_0^T \|D^{3/2}[D^{-s};G]u^n\|^2 \, dt
	&\leq c\int_0^T \|D^{-s+1/2}u^n\|^2 \, dt \\
	&\leq c(T)\|u_T^n\|_{-s-1}^2
\end{align*}
by the propagation of regularity result for IVP \eqref{E:KDV5CL*}.
Therefore
\begin{equation*}
\int_0^T \|GD^{-s}u^n\|_{3/2}^2 \, dt \rightarrow 0
	\qquad\text{as}\qquad
		n\rightarrow\infty.
\end{equation*}
Inserting this into \eqref{OBS4_A} and using that
$u^n\rightarrow0$ in $Z_{-s-1,T}$ we conclude that
$u_T^n\rightarrow0$ in $H_0^{-s}(\T)$, thus
contradicting the fact that $\|u_T\|_{-s}=\|u_T^n\|_{-s}=1$.	
Thus \eqref{OBS4_B} holds.

We now show that \eqref{OBS4_B} implies \eqref{OBS4}.
To obtain a contradiction, suppose there is a sequence
$\{u_T^n\}$ in $H_0^{-s}(\T)$ such that
\begin{equation} \label{OBS4_F}
1 = \|u_T^n\|_{-s}^2
	> n \left( \int_0^T \|D^{3/2}(Gu^n)\|_{-s}^2 \, dt \right),
\end{equation}
with $u^n$ denoting the solution to \eqref{E:KDV5CL*}
corresponding to $u(T)=u_T^n$.
By compactness of the embedding
$H_0^{-s}(\T) \hookrightarrow H_0^{-s-1}(\T)$ then
$u_T^n \rightarrow u_T$ in $H_0^{-s-1}(\T)$.
Applying \eqref{OBS4_B} to the difference of two solutions,
\begin{align*}
\|u_T^n-u_T^m\|_{-s}^2
	&\leq C\left( \int_0^T \|D^{3/2}G(u^n-u^m)\|_{-s}^2 \, dt
				+ \|u_T^n-u_T^m\|_{-s-1}^2\right) \\
	&\leq C\left( \int_0^T \|D^{3/2}(Gu^n)\|_{-s}^2 \, dt
				+ \int_0^T \|D^{3/2}(Gu^m)\|_{-s}^2 \, dt\right) \\
	&\qquad\qquad
			+ C\|u_T^n-u_T^m\|_{-s-1}^2.
\end{align*}
Combining this with \eqref{OBS4_F} implies that
$u_T^n \rightarrow u_T$ strongly in $H_0^{-s}(\T)$.
Letting $n\rightarrow \infty$ in \eqref{OBS4_F} we have that
\begin{equation*}
\int_0^T \|D^{3/2}(Gu)\|^2 \, dt = 0,
\end{equation*}
with $u$ denoting the solution to \eqref{E:KDV5CL*}
corresponding to $u(T)=u_T$.
Hence $Gu=0$ a.e. $\T\times(0,T)$ and an application of
Proposition \ref{P:UCP} implies that $u\equiv 0$, thus
contradicting the fact that $\|u_T\|_{-s}=\|u_T^n\|_{-s}=1$.	
Thus \eqref{OBS4} holds.
\end{proof}

We are now able to prove Theorem \ref{T:CONT}, local exact
control of the nonlinear equation \eqref{E:K5F}.
As in the remarks following the proof of Theorem \ref{T:STAB},
the results in this section apply to equation \eqref{E:K5F}
as well as a class of equations containing the KdV hierarchy.
\begin{proof}
For each $s\geq0, T>0$, Lemma \ref{P:LEC} provides the
existence of a continuous linear operator \cite[Lemma
2.48, p. 58]{MR2302744}
\begin{equation*}
\Lambda : H_0^s(\T) \rightarrow L^2(0,T;H_0^s(\T))
\end{equation*}
such that given $v_T \in H_0^s(\T)$, the solution $v$
of \eqref{E:KDV5CL} associated to $v_0=0$ and $k=\Lambda(v_T)$
satisfies $v(T)=v_T$.
Denote this solution by
\begin{equation*}
W(k)(t) = v(t) = \int_0^T \mathcal{S}(t-t')GD^{3/2}k(t') \, dt'.
\end{equation*}
From Proposition \ref{P:BSA}, it holds that
$W : L^2(0,T;H_0^s(\T)) \rightarrow Z_{s,T}$ is continuous.

Let $u_0,u_T \in H_0^s(\T)$, $s>2$, with
\begin{equation*}
\|u_0\|_s\leq\rho
	\qquad\text{and}\qquad
		\|u_T\|_s\leq\rho,
\end{equation*}
for some $\rho>0$ to be determined.
For $v \in Z_{s,T}$, set
\begin{equation*}
\omega(v) = \int_0^t \mathcal{S}(t-t')(v\partial_x^3v)(t') \, dt'
\end{equation*}
and note that Proposition~\ref{P:BSA} yields
\begin{equation} \label{OMEGA}
\|\omega(v)\|_s
	\leq c\left(\int_0^T\|v\partial_x^3v\|_{s-3/2}^2\right)^{1/2}
	\leq c\|v\|_{s,T}^2,
\end{equation}
by the algebra property of $H^{s-3/2}(\T)$ for $s>2$.
Thus $\omega(v)(T) \in H_0^s(\T)$.
Defining
\begin{align*}
\Gamma(v)
	&:= \mathcal{S}(t)u_0
			- \int_0^t \mathcal{S}(t-t')(v\partial_x^3v)(t') \, dt' \\
	&\qquad\qquad
			+ W\left(\Lambda(u_T-\mathcal{S}(T)u_0+\omega(v)(T))\right),
\end{align*}
it is clear that $\Gamma(v)(0)=u_0$ and $\Gamma(v)(T)=u_T$ for
any $v \in Z_{s,T}$.
Thus it suffices to establish a fixed point of the nonlinear
map $v\mapsto\Gamma(v)$ in a closed ball in $Z_{s,T}$.

Repeating the argument of the proof of Theorem~\ref{T:LWP_REG},
we show that $\Gamma$ defines a contraction on
$B = \{ v \in Z_{s,T} : \|v\|_{s<T} \leq R \}$ for
appropriate choices of $R>0$ and $\rho>0$.
The estimate \eqref{UNIFORM} yields
\begin{align*}
\|\Gamma(u)\|_{s,T}
	&\leq c(s,T)\left( \|u_0\|_s
				+ \|u\partial_x^3u\|_{L^2(0,T;H^{s-3/2}(\T))} \right) \\
	&\qquad\qquad
				+ \|W\left(\Lambda(u_T-\mathcal{S}(T)u_0+\omega(v)(T))\right)\|_{s,T}.
\end{align*}
Assuming $u \in Z_{s,T}$, then by the algebra property of
$H_0^{s-3/2}(\T)$,
\begin{align*}
\int_0^T \|u\partial_x^3u\|_{s-3/2}^2 \, dt
	&\leq c\|u\|_{s,T}^4.
\end{align*}
Estimate \eqref{OMEGA}, along with the continuity of
$\Gamma$ and $W$, yields
\begin{equation*}
\|W\left(\Lambda(u_T-\mathcal{S}(T)u_0+\omega(v)(T))\right)\|_{s,T}
	\leq C_0(\|u_0\|_s+\|u_T\|_s) + C_1\|u\|_{s,T}^2
\end{equation*}
Therefore
\begin{equation*}
\|\Gamma(u)\|_{s,T} \leq C_0(\|u_0\|_s+\|u_T\|_s) + C_1\|u\|_{s,T}^2
\end{equation*}
for some $C_0,C_1>0$.
Similarly,
\begin{equation*}
\|\Gamma(u)-\Gamma(v)\|_{s,T}
	\leq C_1\left(\|u\|_{s,T}+\|v\|_{s,T}\right)\|u-v\|_{s,T}.
\end{equation*}
Thus $\Gamma$ forms a contraction on $B$ provided
\begin{equation*}
C_0(\|u_0\|_s+\|u_T\|_s) + C_1R^2 < R
	\quad\text{and}\quad
		2C_1R<1.
\end{equation*}
It is sufficient to take
\begin{equation*}
R=(4C_1)^{-1}
	\quad\text{and}\quad
		\|u_0\|_s \leq \rho := R/2C_0.
\end{equation*}
\end{proof}

\end{section}

\vskip3mm
\noindent {\bf Acknowledgments.}
The authors would like to thank Prof. Felipe Linares
and Prof. Lionel Rosier for reading a draft of this work.

\def\cprime{$'$}


\begin{thebibliography}{10}

\bibitem{MR2834850}
F.~D. Araruna, R.~A. Capistrano-Filho, and G.~G. Doronin.
\newblock Energy decay for the modified {K}awahara equation posed in a bounded
  domain.
\newblock {\em J. Math. Anal. Appl.}, 385(2):743--756, 2012.

\bibitem{MR0463715}
D.~J. Benney.
\newblock A general theory for interactions between short and long waves.
\newblock {\em Studies in Appl. Math.}, 56(1):81--94, 1976/77.

\bibitem{MR0385355}
J.~L. Bona and R.~Smith.
\newblock The initial-value problem for the {K}orteweg-de {V}ries equation.
\newblock {\em Philos. Trans. Roy. Soc. London Ser. A}, 278(1287):555--601,
  1975.

\bibitem{MR1209299}
J.~Bourgain.
\newblock Fourier transform restriction phenomena for certain lattice subsets
  and applications to nonlinear evolution equations. {I}. {S}chr\"odinger
  equations.
\newblock {\em Geom. Funct. Anal.}, 3(2):107--156, 1993.

\bibitem{MR1215780}
J.~Bourgain.
\newblock Fourier transform restriction phenomena for certain lattice subsets
  and applications to nonlinear evolution equations. {II}. {T}he
  {K}d{V}-equation.
\newblock {\em Geom. Funct. Anal.}, 3(3):209--262, 1993.

\bibitem{MR1969209}
J.~Colliander, M.~Keel, G.~Staffilani, H.~Takaoka, and T.~Tao.
\newblock Sharp global well-posedness for {K}d{V} and modified {K}d{V} on
  {$\Bbb R$} and {$\Bbb T$}.
\newblock {\em J. Amer. Math. Soc.}, 16(3):705--749, 2003.

\bibitem{MR2302744}
J.-M. Coron.
\newblock {\em Control and nonlinearity}, volume 136 of {\em Mathematical
  Surveys and Monographs}.
\newblock American Mathematical Society, Providence, RI, 2007.

\bibitem{MR2177163}
W.~Craig, P.~Guyenne, and H.~Kalisch.
\newblock Hamiltonian long-wave expansions for free surfaces and interfaces.
\newblock {\em Comm. Pure Appl. Math.}, 58(12):1587--1641, 2005.

\bibitem{MR2253466}
B.~Dehman, P.~G\'erard, and G.~Lebeau.
\newblock Stabilization and control for the nonlinear {S}chr\"odinger equation
  on a compact surface.
\newblock {\em Math. Z.}, 254(4):729--749, 2006.

\bibitem{MR2434910}
G.~G. Doronin and N.~A. Larkin.
\newblock Kawahara equation in a bounded domain.
\newblock {\em Discrete Contin. Dyn. Syst. Ser. B}, 10(4):783--799, 2008.

\bibitem{MR3462574}
G.~Gao and S.-M. Sun.
\newblock A {K}orteweg--de {V}ries type of fifth-order equations on a finite
  domain with point dissipation.
\newblock {\em J. Math. Anal. Appl.}, 438(1):200--239, 2016.

\bibitem{Gardner1967}
C.~S. Gardner, J.~M. Greene, M.~D. Kruskal, and R.~M. Miura.
\newblock Method for solving the {K}orteweg-de {V}ries equation.
\newblock {\em Phys. Rev. Lett.}, 19(19):1095--1097, 1967.

\bibitem{MR2569891}
O.~Glass and S.~Guerrero.
\newblock On the controllability of the fifth-order {K}orteweg-de {V}ries
  equation.
\newblock {\em Ann. Inst. H. Poincar\'e Anal. Non Lin\'eaire},
  26(6):2181--2209, 2009.

\bibitem{MR2653659}
A.~Gr\"unrock.
\newblock On the hierarchies of higher order m{K}d{V} and {K}d{V} equations.
\newblock {\em Cent. Eur. J. Math.}, 8(3):500--536, 2010.

\bibitem{MR2531556}
Z.~Guo.
\newblock Global well-posedness of {K}orteweg-de {V}ries equation in
  {$H^{-3/4}(\Bbb R)$}.
\newblock {\em J. Math. Pures Appl. (9)}, 91(6):583--597, 2009.

\bibitem{MR3096990}
Z.~Guo, C.~Kwak, and S.~Kwon.
\newblock Rough solutions of the fifth-order {K}d{V} equations.
\newblock {\em J. Funct. Anal.}, 265(11):2791--2829, 2013.

\bibitem{MR610244}
D.~Henry.
\newblock {\em Geometric theory of semilinear parabolic equations}, volume 840
  of {\em Lecture Notes in Mathematics}.
\newblock Springer-Verlag, Berlin-New York, 1981.

\bibitem{MR759907}
T.~Kato.
\newblock On the {C}auchy problem for the (generalized) {K}orteweg-de {V}ries
  equation.
\newblock In {\em Studies in applied mathematics}, volume~8 of {\em Adv. Math.
  Suppl. Stud.}, pages 93--128. Academic Press, New York, 1983.

\bibitem{MR2767079}
T.~Kato.
\newblock Local well-posedness for {K}awahara equation.
\newblock {\em Adv. Differential Equations}, 16(3-4):257--287, 2011.

\bibitem{MR2989690}
T.~Kato.
\newblock Global well-posedness for the {K}awahara equation with low
  regularity.
\newblock {\em Commun. Pure Appl. Anal.}, 12(3):1321--1339, 2013.

\bibitem{Kawahara1972}
T.~Kawahara.
\newblock Oscillatory solitary waves in dispersive media.
\newblock {\em Journal of the Physical Society of Japan}, 33(1):260--264, 1972.

\bibitem{MR3301874}
C.~E. Kenig and D.~Pilod.
\newblock Well-posedness for the fifth-order {K}d{V} equation in the energy
  space.
\newblock {\em Trans. Amer. Math. Soc.}, 367(4):2551--2612, 2015.

\bibitem{MR3513119}
C.~E. Kenig and D.~Pilod.
\newblock Local well-posedness for the {K}d{V} hierarchy at high regularity.
\newblock {\em Adv. Differential Equations}, 21(9-10):801--836, 2016.

\bibitem{MR1211741}
C.~E. Kenig, G.~Ponce, and L.~Vega.
\newblock Well-posedness and scattering results for the generalized
  {K}orteweg-de {V}ries equation via the contraction principle.
\newblock {\em Comm. Pure Appl. Math.}, 46(4):527--620, 1993.

\bibitem{MR1195480}
C.~E. Kenig, G.~Ponce, and L.~Vega.
\newblock Higher-order nonlinear dispersive equations.
\newblock {\em Proc. Amer. Math. Soc.}, 122(1):157--166, 1994.

\bibitem{MR1321214}
C.~E. Kenig, G.~Ponce, and L.~Vega.
\newblock On the hierarchy of the generalized {K}d{V} equations.
\newblock In {\em Singular limits of dispersive waves ({L}yon, 1991)}, volume
  320 of {\em NATO Adv. Sci. Inst. Ser. B Phys.}, pages 347--356. Plenum, New
  York, 1994.

\bibitem{MR2501679}
N.~Kishimoto.
\newblock Well-posedness of the {C}auchy problem for the {K}orteweg-de {V}ries
  equation at the critical regularity.
\newblock {\em Differential Integral Equations}, 22(5-6):447--464, 2009.

\bibitem{MR3473453}
C.~Kwak.
\newblock Local well-posedness for the fifth-order {K}d{V} equations on
  {$\Bbb{T}$}.
\newblock {\em J. Differential Equations}, 260(10):7683--7737, 2016.

\bibitem{MR2455780}
S.~Kwon.
\newblock On the fifth-order {K}d{V} equation: local well-posedness and lack of
  uniform continuity of the solution map.
\newblock {\em J. Differential Equations}, 245(9):2627--2659, 2008.

\bibitem{MR2654198}
C.~Laurent.
\newblock Global controllability and stabilization for the nonlinear
  {S}chr\"odinger equation on an interval.
\newblock {\em ESAIM Control Optim. Calc. Var.}, 16(2):356--379, 2010.

\bibitem{MR2644360}
C.~Laurent.
\newblock Global controllability and stabilization for the nonlinear
  {S}chr\"odinger equation on some compact manifolds of dimension 3.
\newblock {\em SIAM J. Math. Anal.}, 42(2):785--832, 2010.

\bibitem{MR2753618}
C.~Laurent, L.~Rosier, and B.-Y. Zhang.
\newblock Control and stabilization of the {K}orteweg-de {V}ries equation on a
  periodic domain.
\newblock {\em Comm. Partial Differential Equations}, 35(4):707--744, 2010.

\bibitem{MR0235310}
P.~D. Lax.
\newblock Integrals of nonlinear equations of evolution and solitary waves.
\newblock {\em Comm. Pure Appl. Math.}, 21:467--490, 1968.

\bibitem{MR3335395}
F.~Linares and L.~Rosier.
\newblock Control and stabilization of the {B}enjamin-{O}no equation on a
  periodic domain.
\newblock {\em Trans. Amer. Math. Soc.}, 367(7):4595--4626, 2015.

\bibitem{MR755731}
P.~J. Olver.
\newblock Hamiltonian and non-{H}amiltonian models for water waves.
\newblock In {\em Trends and applications of pure mathematics to mechanics
  ({P}alaiseau, 1983)}, volume 195 of {\em Lecture Notes in Phys.}, pages
  273--290. Springer, Berlin, 1984.

\bibitem{MR2446185}
D.~Pilod.
\newblock On the {C}auchy problem for higher-order nonlinear dispersive
  equations.
\newblock {\em J. Differential Equations}, 245(8):2055--2077, 2008.

\bibitem{MR1216734}
G.~Ponce.
\newblock Lax pairs and higher order models for water waves.
\newblock {\em J. Differential Equations}, 102(2):360--381, 1993.

\bibitem{MR2565262}
L.~Rosier and B.-Y. Zhang.
\newblock Control and stabilization of the {K}orteweg-de {V}ries equation:
  recent progresses.
\newblock {\em J. Syst. Sci. Complex.}, 22(4):647--682, 2009.

\bibitem{MR2486102}
L.~Rosier and B.-Y. Zhang.
\newblock Local exact controllability and stabilizability of the nonlinear
  {S}chr\"odinger equation on a bounded interval.
\newblock {\em SIAM J. Control Optim.}, 48(2):972--992, 2009.

\bibitem{MR1360229}
D.~L. Russell and B.-Y. Zhang.
\newblock Exact controllability and stabilizability of the {K}orteweg-de
  {V}ries equation.
\newblock {\em Trans. Amer. Math. Soc.}, 348(9):3643--3672, 1996.

\bibitem{MR871574}
J.-C. Saut and B.~Scheurer.
\newblock Unique continuation for some evolution equations.
\newblock {\em J. Differential Equations}, 66(1):118--139, 1987.

\bibitem{MR761761}
M.~Schwarz, Jr.
\newblock The initial value problem for the sequence of generalized
  {K}orteweg-de {V}ries equations.
\newblock {\em Adv. in Math.}, 54(1):22--56, 1984.

\bibitem{MR2775188}
C.~F. Vasconcellos and P.~N. da~Silva.
\newblock Stabilization of the {K}awahara equation with localized damping.
\newblock {\em ESAIM Control Optim. Calc. Var.}, 17(1):102--116, 2011.

\bibitem{MR3356494}
X.~Zhao and B.-Y. Zhang.
\newblock Global controllability and stabilizability of {K}awahara equation on
  a periodic domain.
\newblock {\em Math. Control Relat. Fields}, 5(2):335--358, 2015.

\end{thebibliography}
\end{document}